\documentclass[11pt, fleqn]{amsart}
\usepackage{amsmath, amssymb, amsthm} 
\usepackage{hyperref}
\usepackage{mathtools}
\usepackage{graphicx}
\usepackage{adjustbox}
\usepackage{caption}
\usepackage{subcaption}
\usepackage{booktabs}
\usepackage{float}
\usepackage{longtable}
\usepackage{listings}
\usepackage{xcolor}
\newtheorem{theorem}{Theorem}[section]
\newtheorem{lemma}[theorem]{Lemma}

\newtheorem{proposition}{Proposition}
\newtheorem*{proposition*}{Proposition}

\theoremstyle{remark}
\newtheorem{remark}{Remark}
\title
{Ramanujan functions on small primes
}
\author{Barry Brent}
\email{axcjh@bu.edu}
\subjclass[2020]{05A17, 05C62, 11F11, 11P83}
\keywords{partitions, graphs, cusp forms}
\begin{document}
\begin{abstract}  Lehmer asked whether Ramanujan's tau
function has any zeros and showed that, if it does,
the smallest $n$ such that
$\tau(n) = 0$ is a prime number. 
Certain matrix identities  originating
in the theory of symmetric functions express
members of arithmetic sequences as determinants. We study empirically the eigenvalues of the matrices underlying them with the goal of estimating
how close the eigenvalues of the matrix associated to the number $\tau(p)$ for $p$ prime get to zero,
because (with our set-up) a zero 
eigenvalue would entail that
$\tau(p)$ vanishes. We also studied 
corresponding eigenvalues
of several cusp forms coming from elliptic curves.
\end{abstract}
\maketitle
\vspace{-0.5em}
\begin{center}
    \rm{10 March 2026 no4}
\end{center}
\vspace{1em}
\section{Introduction}
Lehmer asked whether Ramanujan's tau
function has any zeros, and he showed that, if it does,
the smallest $n$ such that
$\tau(n) = 0$ is a prime number \cite{Le47}. 
Matrix identities (\cite{M}, Lemmas 2.1 and 2.2 below) govern certain additive convolution identities. We study eigenvalues of these matrices empirically,
especially eigenvalues associated
to the
function $n \mapsto \tau(p_n)$ where $p_n$ denotes the $n^{th}$ prime,
because (with our set-up) a zero 
eigenvalue would entail that
$\tau(p_n)$ vanishes.
After a procedure we will call ``deformation'', the minimum modulus
among the eigenvalues for a given $n$ appears to oscillate in a nearly periodic fashion with $n$.
(We also saw this phenomenon in the sequence $\{\tau(p_{n+1} - p_n)\}_{n=1, 2, ...}$; we omit it here.) We have also observed oscillatory behavior 
in several newforms coming from elliptic curves in virtue of the Modularity Theorem.
To encompass those observations, given the appropriate uniformizing function $q_*(z) = q_*$  (say) from the upper half plane $\mathbb{H}$ to the unit disk for a cusp form $f(z) = \sum_n a(n) q_*^n$, we have referred in our title to the function $n \mapsto a(n)$ as a \emph{Ramanujan function}.
\section{Lemmas}
The lemmas in this
section are well known. 
They appear, for example, in MacDonald's book \cite{M}), 
but the proofs
are not easy to locate. (In fact, we
have not located them.)
We have placed home-made
proofs of these statements in
our repository \cite{Bre25}\footnote{``Lemmas for `Ramanujan's function on small primes' ''}.

In the sequel, we denote the cardinality of a set $S$ by 
$\#S$ and the determinant of a matrix $M$ by $|M|$.
\begin{lemma}
The equations  below are equivalent. (We will
refer to both of them as equation (D) in the sequel.)
Let $h_0=1$ and
\[
n h_n = \sum_{r=1}^{n} j_r h_{n-r}\tag{D1}
\]
for $n \ge 1$.
With
$
H(t) = \sum_{n=0}^\infty h_n t^n$ and  $J(t) = \sum_{r=1}^\infty j_r t^r$,
\[
t \frac{d}{dt}H(t) = H(t) J(t).\tag{D2}\]
\end{lemma}
With $f_0$ possibly defined as equal to one,,    
let \(\overline{f}\) denote the sequence 
\(\{f_n\}_{n \in \mathbb{Z}^+}\).
Let $J_n(\overline{j})$ and $H_n(\overline{h})$ be the matrices 
\[ 
\left(
\begin{matrix}
j_1 & -1 & 0 & \cdots & 0 \\
j_2 & j_1 & -2 & \cdots & 0 \\
\vdots & \vdots & \vdots & \vdots & \vdots \\
j_{n-1} & j_{n-2} &  \cdots & j_1 & -n + 1\\
j_n & j_{n-1} & j_{n-2} & \cdots & j_1\\
\end{matrix}
\right )
\]
and
\[ 
\left(
\begin{matrix}
h_1 & 1 & 0 & \cdots & 0 \\
2h_2 & h_1 & 1 & \cdots & 0 \\
\vdots & \vdots & \vdots & \vdots & \vdots \\
nh_n & h_{n-1} & h_{n-2} & \cdots & h_1\\
\end{matrix}
\right),
\]
respectively.
\begin{lemma}
Let $h_0= j_1 = 1$.
    Equation (D) implies
that
\begin{equation}
j_n = (-1)^{n+1} | H_n(\overline{h})|. 
\end{equation}
and
\begin{equation}
n! h_n =  | J_n(\overline{j})|
\end{equation}
for $n \ge 1$.
\end{lemma}

Next we state a corollary of the
following theorem in Vein and Dale's book 
~\cite{Vein1999}. 
\begin{proposition}
(Vein and Dale, Theorem 4.23)
Let
\(A_n =\)
\[
\begin{pmatrix}
a_1 & -1 & 0 & \cdots & 0 \\
a_2 & a_1 & -2 & \cdots & 0 \\
\vdots & \vdots & \vdots & \ddots & \vdots \\
a_{n-1} & a_{n-2} & \cdots & a_1 & -(n{-}1) \\
a_n & a_{n-1} & \cdots &a_2& a_1
\end{pmatrix}^T,
\]
where the $T$ operator is matrix transpose.
Let \(B_n(x) = |A_n -xI_n|\),
so that 
\(B_n(x) = (-1)^n \chi_{A_n}(x)\). 
 Then 
\[
B_n(x) = \sum_{r=0}^n 
\binom{n}{r}
|A_r|x^{n-r}.
\]
\end{proposition}
The corollary is the following
\begin{theorem}\footnote{Here we are merely codifying an
argument shown to us by
J. L\'opez-Bonilla.}
With $h$ and $j$ as in the above lemmas,
\[\chi({(J_n(\overline{j}))}(x)=
\sum_{r=0}^n 
\binom{n}{r}
r!  h_r (-1)^r x^{n-r}.
\]
\end{theorem}
\begin{proof}
Determinants are invariant under the transpose,
so, setting each $a_i$  equal to $j_i$ in the
lemmas above, the proposition can be restated in the following way:
Let
\(J_n(\overline{j}) =\)
\[
\begin{pmatrix}
j_1 & -1 & 0 & \cdots & 0 \\
j_2 & j_1 & -2 & \cdots & 0 \\
\vdots & \vdots & \vdots & \ddots & \vdots \\
j_{n-1} & a_{n-2} & \cdots & j_1 & -(n{-}1) \\
j_n & j_{n-1} & \cdots &j_2& j_1
\end{pmatrix}.
\]
Let \(C_n(x) = |(J_n(\overline{j}) -xI_n|\),
so that 
\(C_n(x) = (-1)^n \chi({(J_n(\overline{j}))}(x)\). 
 By Proposition 1, $C_n(x) =$  
\[
 \sum_{r=0}^n 
\binom{n}{r}
|(J_r(\overline{j})|(-1)^{n-r} x^{n-r} = \]
(by Lemma 2.2)
\[
 \sum_{r=0}^n 
\binom{n}{r}
r!  h_r(-x)^{n-r}
= \sum_{r=0}^n 
\binom{n}{r}
r!  h_r (-1)^{n-r}x^{n-r}.
\]
Simplifying, 
\(
\chi({(J_n(\overline{j}))}(x)=
\sum_{r=0}^n 
\binom{n}{r}
r!  h_r (-1)^r x^{n-r}.
\)
\end{proof}
Now let $M$ be an $n \times n$ matrix over a field.
 For any subset $S \subseteq \{1, 2, \ldots, n\}$, let $M[S,S]$ denote the submatrix of $M$ with rows and columns indexed by $S$. The following proposition is 
well known. (For example, see equation (1.2.13) in Horn and Johnson's book ~\cite{Horn2012}.)
\begin{proposition}
The coefficient $a_k$ of $x^k$ in $\chi(M)$
satisfies
\[
a_{n-k} = (-1)^k \sum_{\substack
{\#S = k\\{S \subseteq 
\{1, 2, \ldots, n\}}}} |M[S,S]|.
\]
\end{proposition}
\section[]{The sequences
$\{p_n\}_{n \ge 1}$
and
$\{\tau(p_n)\}_{n \ge 1}$}
In this section
and the next, we write 
$J^{(c)}_n(\overline{j})$ and 
$H^{(c)}_n(\overline{h})$ 
for 
the matrices
$J_n\left(\{c, j_1, j_2, ...\}\right)$ and
$H_n\left(\{c, h_1, h_2, ...\}\right)$
and refer to them as deformations of
the matrices 
 $H_n(\overline{h})$ 
and $J_n(\overline{j})$.

We will freely abuse the notation by 
employing the $H, h, J$, and $j$ symbols
to denote objects in a way that depends on the context
without, we hope, confusing anyone.
\subsection[]{The sequence
$\{\tau(p)\}_{p\text{ prime}}$ and its developments}
Let $p_n$ denote the $n^{th}$ prime, $h(0) = 1$ and $h(n)= \tau(p_n) =$ (say) $\tau_p(n)$ for positive integers $n$. Let  $j(n)$ and $h(n)$ together satisfy equation (D). Given $h$, the elements of 
$\overline{j} =
\{j\}_{n = 1, 2, ...}$ 
can be computed from equation (D1). By Lemma 2.2 (2),
$\tau_p(n) = |J_n(\overline{j})|/n!$.
Let $\Pi_n(x) = \chi(J_n(\overline{j}))(x) = |xI - J_n(\overline{j})|$ be the characteristic polynomial of $J_n(\overline{j})$. Plots 3 and 4 show 
the behavior of the minimum modulus
among the eigenvalues of $\Pi^{(1)}_n$
and $\Pi_n$ respectively. They 
are relevant to Lehmer's question
because 
the following statements are 
equivalent:\newline
(a) $\tau(p_n) = 0$, \newline
(b) $|J_n(\overline{j})| = 0$, and
\newline
(c) $\Pi_n(0) = 0$. 
\begin{remark}
Writing\footnote{See `undeformed tauprime2nov25no3.ipynb' in the repository \cite{Bre25}.} $\Pi_n(x)
= \sum_{k=0}^n a_{n,n-k} x^{n-k}
$, Theorem 2.3 gives
\begin{equation}
a_{n,n-k} = (-1)^k k! \binom{n}{k}
\tau_p(k).
\end{equation}
This result 
allows bypass of the 
construction of the matrices $J_n(\overline{j})$,
in computations, saving time and allowing
the extension our data sets. We have 
not exploited this advantage; if
a similar result applied to the deformations
$\chi(J_n^{(c)}(\overline{j}))(x)$ (see section 3.2) we would 
have used it but, at present, this
is not available to us.
\end{remark}
\begin{remark}
The number of terms in the sum in Proposition 2
is $\binom{n}{k}$, so the simplest way to reconcile
Proposition 2 and equation (3) is to suppose that the structural symmetries of 
$J_n(\overline{j_{\tau_p}})$ somehow force
$|J_n(\overline{j_{\tau_p}})[S,S]|= 
(-1)^k k! \tau_p(k)$ just when $\#S = k$ and $S \subseteq 
\{1, 2, \ldots, n\}$. But this turns out to be too 
good to be true.\footnote{See `too good.ipynb'in \cite{Bre25}.} Instead, there must be
systematic cancellations
coming from those symmetries.
So far, we have not understood these cancellations.
\end{remark}
\subsection{Characteristic
functions of the deformed matrices
$J_n^{(c)}(\overline{j})$ representing the sequence
$\{\tau(p_n)\}_{n \ge 1}$}
Let
$
\chi(J_n^{(1)}(\overline{j_{\tau_p}})) =
\Pi_n^{(1)}(x)
$ (say.) (The corresponding objects for $c=0$ exhibit 
behaviors similar to those we go over here.\footnote{See  `deformed primetau 22oct25.ipynb' in our repository \cite{Bre25}.})
Let $V_n^{(1)}$ be the 
set of zeros of $\Pi_n^{(1)}(x)$ and let $\mu_n^{(1)} = \min_{v \in V_n^{(1)}} |v|$; then it appears that the graph of the pairs $(n,\mu_n^{(1)})$ lies on 
an oscillating curve. As a control, we repeat the calculation for $h(n) = \tau(p_n+ 1)$. Figure 2 shows no oscillatory behavior for this choice of $h$, indicating that the choice of prime indices \it{per se} \rm was decisive.

The deep spikes in the lower, logarithmic plot in Figure 1 suggest that the approach of the roots of $\Pi_n^{(1)}(x)$ to the origin may behave in a predictable way, but the details are
not clear to us because (among other reasons) our data is limited. 
Also, we do not know the 
details of the relationship of the roots of 
$\Pi_n(x)$ to those of the $\Pi_n^{(c)}(x)$, so 
the connection, if any, to Lehmer's question is
also unclear.
Furthermore, we have no statement like 
equation (3) for the 
$\Pi_n^{(c)}(x)$, because we do not 
have an explicit description of the 
sequence $\{h_n\}$ represented in Lemma 2.2
by the matrices
of which the $\Pi_n^{(c)}(x)$ are 
the characteristic polynomials. As we said above,
such a description would permit us to sidestep
the construction of the matrices we currently
use to find these polynomials and the resulting
time savings would in practice allow us to
collect more data. More data would be welcome
in view of the oscillations we are trying to check.

The lower envelope of the lower  plot of logarithms of minimum moduli for each $n$ in $\{2, 3, ..., 400\}$ displayed in Figure 3 
appear to behave regularly: the abscissas of the points 
on the lower envelope form a set $\{5, 9, ..., 395\}$, which, when they are replaced by their residues modulo $4$, make Table 1.

The size of the blocks of equal residues are $10$ for the first one and $9$ for each of the others.\footnote{We are using data from the lower envelope of the logarithms of the minimum moduli. (The lower plot in Figure 3.)  See ``tauprime=h logarithmic envelopes prec100 n400  1dec25.ipynb'' in \cite{Bre25}.} 

(For contrast, we reproduce in Figure 4 a plot 
of the mimimum moduli for each $n$ for 
the undeformed characteristic  
polynomials $\Pi_n(x)$. We have not attempted the sort of analysis 
we just discussed for the $\Pi_n^{(1)}(x)$.) 

Fourier analyses were done by the Anthropic AI ``Claude'' 
\cite{Anthropic2026}. (We will supply more details and also put Claude's analyses verbatim in our depository in later drafts.)

Next, with $h_n = \tau_p(n)$, we used equation (D1) to compute  the sequence $\overline{j_{\tau_p}}$. We then ``deformed''
it by prepending $c$ to $\overline{j_{\tau_p}}$. We denote this deformation as $\overline{j^{(c)}_{\tau_p}}$.
Then, once more applying equation (D1), we compute a sequence $\overline{h^{(c)}_{\tau_p}}$ from $\overline{j^{(c)}_{\tau_p}}$. We formed the list of pairs 
$(c, n! h^{(c)}(n))$. (These are just the values of the determinants from Lemma 2.2 in this situation; but it is a situation in which we can bypass calculating the determinants and, by that expedient, achieve significant time savings.)

It seemed likely that these pairs would describe a polynomial in $c$, $D_n(c)$ (say.) We found by Lagrange interpolation this is so within the range of our observations, and
that the degree of $D_n(c)$ appears to be $n$. 

We found the minimum modulus among the roots of $D_n(c)$
and plotted these minima against $n$. In the case $c \in \mathbb{Z}^{\geq 0}$, the  result was
the same kind of oscillating wave form, as far as we could see (literally.)\footnote{Figure 5,  Table 7,
 and notebook `tauprime many c's 18feb26'
in \cite{Bre25}. The reader will notice that we ran the routines twice, at 64-bit and 100-bit precision; the results were indistinguishable.}
The Fourier analysis of this waveform is described in Table 2. 
\section{Oscillatory behavior in newforms from elliptic curves} 
Because the cusp form $\Delta$ is the generating function of the multiplicative Ramanujan function, which displays oscillatory behavior in the range of our observations, we searched a source of cusp forms that are generating functions of multiplicative functions, and found similar oscillatory behavior: Cremona's database \cite{LMFDB} of cusp forms $\sum_n a(n) q_*^n $ for an appropriate uniformizing parameter $q_*(z)$ associated to elliptic curves in virtue of the Modularity Theorem. 

These cusp forms have weight two at various 
levels.\footnote{SageMath notebooks corresponding to Table 4 are in \cite{Bre25}.} In the limited range of our observations, some deformed curves (in the sense we have been using that term) show oscillatory waveforms, and some that do not, do show oscillatory behavior when the sequences $\{h\}$
satisfy $h(n) = a(1+p_n)$.\footnote{We had thought, mistakenly, that our code was checking 
the coefficients $a(p_n)$. To assure ourselves that our present interpretation of the code is the correct one, we recomputed some of the plots with code that checks it explicitly. For example, see the notebooks ``\texttt{curve 11a1 $h(n)=a(p_n+1)$ 8mar26}'' and ``\texttt{curve 14a1 $h(n)=a(p_n+1)$ 9mar26}'' in \cite{Bre25}.}
Table 3 summarizes Claude's 
findings.\footnote{The analysis of the curve 
`$\tau(p_n), c = 0$ was carried out in the notebook `deformed c=0 primetau3feb26' and that of the curve ``$\tau(p_n), c = 1$, was done in the notebook `tauprime=h logarithmic envelopes prec100 n400 1dec25', both stored in \cite{Bre25}.} The curve labels indicate characteristics of the associated elliptic curves. (See the explanation in section 6.1 below of Table 8, column A.)   The column labeled ``$c$'' indicates the deformation parameter in the sense described above. The column labeled ``Indices'' distinguishes between matrices for $h(n) = a_n$ and
$h(n) = a_{p_n}$.
More detail will follow in subsequent drafts of this article.

\section{Questions}
We have a  dilemma. The plots for the 
undeformed matrices are a mess,
therefore,
unpromising as sources of
insight about the distances of their eigenvalues
from the origin,
but their determinants represent the underlying
$h$ sequences in a straightforward way.
The plots for the deformed matrices
seem more promising.
It is possible that eventually we will be able to use them to formulate some guesses bounding the
plotted minimum modulus eigenvalues away from zero.
But, so far, we cannot use such conjectures
to bound their plotted minimal modulus eigenvalues in the undeformed setting
away from zero, and that is what we need.
Beyond the matter of making guesses, we 
do not know why the oscillations we see
in the deformed setting occur, why they
are absent in the undeformed setting,
or very much about how the two settings are related.
Other questions are:
Are the members of $\overline{h^{(c)}_{\tau_p}}$  Fourier coefficients of a modular form? What conditions on $h(n)$ 
 produce oscillatory behavior among the $h^{(c)}$?

To help with comparisons between the presence or absence of oscillatory behavior and characteristics of the objects that do or do not manifest it, we list cusp forms and associated invariants computed using  SageMath  native functions \cite{sagemath} (Tables 8 - 11.) The tables  are explained in the next section.\footnote{See `cusp form gp invariants 24feb26' in \cite{Bre25}. Tables 8 - 11 and the key to them in the next section were created to our specifications by the AI ``Claude''\cite{Anthropic2026}, and edited by the author.}

\section{Key to Tables 8 - 11}
The tables enumerate several invariants for Ramanujan's $\Delta$ function and elliptic curves over $\mathbb{Q}$ with conductor $N < 54$ and unique Fourier expansions in the Cremona database \cite{cremona:database}. (Thus there is only one curve labeled NXm, as explained below, for each NX.) The key omits discussion of the two rightmost columns in each table. 
These are yes or no responses to the questions stated in the table captions. At the time of the present draft, they are entirely subjective.

\subsection{Table 8: Basic Arithmetic Invariants}

\subsubsection*{Column A: Curve Label}
The Cremona label uniquely identifies an elliptic curve over $\mathbb{Q}$ in the form $NXm$, where $N$ is the conductor, $X$ is a letter indicating the isogeny class (a, b, c, \ldots), and $m$ is a number ordering curves within the isogeny class by discriminant. Isogenous curves have the same associated cusp form (``newform'').

For Ramanujan's $\Delta$ function, we use the label ``Delta''.

Code:
\begin{lstlisting}
label = E_cremona.label()  # Line 267
# For Delta: label = "Delta"  # Line 176
\end{lstlisting}

References: The Cremona labeling scheme is documented in \cite{cremona:algorithms}. Isogeny class structure is discussed in \cite[Chapter VII]{silverman:aec}.

\subsubsection*{Column B: Minimal Weierstrass Form}
The minimal (global) Weierstrass model of the curve, reported as the coefficient
vector $[a_1, a_2, a_3, a_4, a_6]$ in the equation
\[
  y^2 + a_1 xy + a_3 y = x^3 + a_2 x^2 + a_4 x + a_6.
\]

Code:
\begin{lstlisting}
E = EllipticCurve(label)        # e.g. EllipticCurve('37a1')
ainvs = list(E.ainvs())         # returns [a1, a2, a3, a4, a6]
\end{lstlisting}

References: Minimal Weierstrass models: \cite[Section VII.1]{silverman:aec}.
Uniqueness: \cite[Proposition VIII.8.2]{silverman:aec}.

\subsubsection*{Column C: Weight}
The weight of the associated modular form.
Code:
\begin{lstlisting}
weight = 2  # All elliptic curves, Line 286
# For Delta: weight = 12  # Line 181
\end{lstlisting}

References: Modularity Theorem: \cite{wiles:fermat, tay-wil, BCDT}. Weight of modular forms: \cite[Chapter 5]{diamond-shurman}.

\subsubsection*{Column D: Torsion Order}
The order $|E(\mathbb{Q})_{\text{tors}}|$ of the finite torsion subgroup of the Mordell-Weil group.
Code:
\begin{lstlisting}
torsion_order = int(E.torsion_order())  # Line 289
# For Delta: torsion_order = "N.A."  # Line 178
\end{lstlisting}

 References: Mazur's theorem \cite{mazur:torsion}. Computational methods: \cite[Algorithm 7.4.12]{cohen:course}.

\subsubsection*{Column E: Torsion Structure}
The group structure of $E(\mathbb{Q})_{\text{tors}}$ as an abelian group, expressed in the form $\mathbb{Z}/n_1\mathbb{Z} \times \cdots \times \mathbb{Z}/n_k\mathbb{Z}$ where $n_i | n_{i+1}$ (invariant factors).

Code:
\begin{lstlisting}
def torsion_structure_string(E):
    T = E.torsion_subgroup()
    inv = T.invariants()  # Returns tuple of invariant factors
    if len(inv) == 0:
        return "trivial"
    else:
        return " x ".join([f"Z/{n}Z" for n in inv])
# Line 101-109
\end{lstlisting}

References: The SageMath method \texttt{torsion\_subgroup()} implements the algorithm from \cite[Section 8.7]{washington:elliptic}.
Structure theorem for finitely generated abelian groups: \cite[Theorem 5.2]{dummit-foote}. Computation: \cite{washington:elliptic}.

\subsubsection*{Column F: Rank}
The rank $r$ of the Mordell-Weil group $E(\mathbb{Q})$.
Code:
\begin{lstlisting}
try:
    rank = int(E.rank())
    rank_certain = True
except Exception:
    rank = -1
    rank_certain = False
# Lines 293-298
\end{lstlisting}

References: SageMath's \texttt{E.rank()} implements the algorithm of \cite{cremona:algorithms}, which combines 2-descent with verification using the BSD conjecture for rank 0 and 1 cases.
Mordell-Weil theorem: \cite{mordell, weil, silverman:aec}. BSD conjecture: \cite{BSD}. Computational methods: \cite{cremona:algorithms}.

\subsection{Table 9: Mordell-Weil Generators and Galois Representations}

\subsubsection*{Column A: Curve Label}
Same as column~A of Table~8.

\subsubsection*{Column B: $j$-Invariant}
The $j$-invariant $j(E) \in \mathbb{Q}$.

Code:
\begin{lstlisting}
E = EllipticCurve(label)
j = E.j_invariant()            # returns an element of QQ
\end{lstlisting}

References: $j$-invariant and the $b$-, $c$-invariants: \cite[Section III.1]{silverman:aec}.
Relation to CM: \cite[Appendix A]{silverman:advanced}.

\subsubsection*{Column C: Generator Canonical Heights}
N\'{e}ron-Tate height $\hat{h}(P)$. 

Code:
\begin{lstlisting}
def generator_heights_string(E):
    try:
        gens = E.gens()
        if len(gens) == 0:
            return "none"
        heights = [float(g.height()) for g in gens]  # CORRECTED
        return "|".join([f"{h:.4f}" for h in heights])
    except Exception as ex:
        return f"error: {ex}"
# Lines 111-119
\end{lstlisting}

References: Canonical height theory: \cite[Chapter VIII]{silverman:aec}. Computational methods: \cite{cremona:algorithms, silverman:computing}.

\subsubsection*{Column D: Regulator}
The determinant of the height pairing matrix.

Code:
\begin{lstlisting}
try:
    reg = float(E.regulator())
except Exception:
    reg = -1.0
# Lines 306-309
\end{lstlisting}

References: Regulator definition: \cite[Section X.6]{silverman:aec}. BSD conjecture: \cite{BSD, gross-zagier}.

\subsubsection*{Column E: Galois Representation Surjective for All Primes?}
For a prime $\ell$, this is the Galois representation
$$\rho_{E,\ell}: \text{Gal}(\overline{\mathbb{Q}}/\mathbb{Q}) \to \text{GL}_2(\mathbb{Z}/\ell\mathbb{Z})$$
arising from the action of the absolute Galois group on the $\ell$-torsion $E[\ell]$. 

Code:
\begin{lstlisting}
try:
    rho = E.galois_representation()
    non_surj = list(rho.non_surjective())
    gal_surjective = (len(non_surj) == 0)
except Exception as ex:
    gal_surjective = "error"
# Lines 313-317
\end{lstlisting}

References: Serre's open image theorem: \cite{serre:galois, serre:abelian}. Computational aspects: \cite{sutherland:galois}.

\subsection{Table 10: Galois Exceptional Primes and Modular Data}
\subsubsection*{Column B: Galois Image Type}
A classification of the image of the mod-$\ell$ Galois representation. Possible types: ``surjective'' ($\text{im}(\rho_{E,\ell}) = \text{GL}_2(\mathbb{Z}/\ell\mathbb{Z})$ for all $\ell$); ``non-surjective at $[\ell_1, \ldots, \ell_k]$'' (list of exceptional primes); or ``CM'' (curve has complex multiplication, image contained in normalizer of split/non-split Cartan subgroup). 
The numbers in brackets are the lists of the finite set of primes $\ell$ for which $\rho_{E,\ell}$ is not surjective. \cite{serre:galois, serre:abelian, sutherland:galois} For non-CM curves over $\mathbb{Q}$, this set is always finite (Serre). 

Code:
\begin{lstlisting}
non_surj = list(rho.non_surjective())
gal_exceptional_primes = "|".join([str(p) for p in non_surj]) \
    if non_surj else "none"
# Lines 318-319
\end{lstlisting}

\begin{lstlisting}
def galois_image_type(E):
    try:
        rho = E.galois_representation()
        if E.has_cm():
            return "CM"
        non_surj = rho.non_surjective()
        if len(non_surj) == 0:
            return "surjective"
        else:
            return f"non-surjective at {list(non_surj)}"
    except Exception as ex:
        return f"error: {ex}"
# Lines 121-131
\end{lstlisting}

References: Image classifications: \cite{serre:galois, serre:abelian, sutherland:galois}
CM theory: 
\cite[Chapter II]{silverman:advanced}.

\subsubsection*{Column C: Index of $\Gamma_0(N)$ in $\text{SL}_2(\mathbb{Z})$}
($N$ is the conductor of the curve.)
Code:
\begin{lstlisting}
def X0N_data(N):
    try:
        G = Gamma0(N)
        idx = G.index()
        genus = G.genus()
        cusps = G.ncusps()
        return idx, genus, cusps
    except Exception as ex:
        return "error", "error", "error"
# Lines 133-141
\end{lstlisting}

References: Congruence subgroups: \cite[Chapter 1]{diamond-shurman}. Index formula: \cite[Proposition 1.4.4]{diamond-shurman}.
\clearpage
\subsection{Table 11: Modular Curves and L-functions}

\subsubsection*{Column B: Genus of $X_0(N)$} $N$ is the conductor.
$X_0(N)$ is the compactification of the quotient $\Gamma_0(N) \backslash \mathbb{H}$ where $\mathbb{H}$ is the upper half-plane. 

Code:
\begin{lstlisting}
G = Gamma0(N)
genus = G.genus()
# Line 136
\end{lstlisting}

References: Modular curves: \cite[Chapter 3]{diamond-shurman}. Genus formula: \cite[Proposition 3.8.2]{diamond-shurman}.

\subsubsection*{Column C: Number of Cusps of $X_0(N)$} 

Code:
\begin{lstlisting}
cusps = G.ncusps()
# Line 137
\end{lstlisting}

References: Cusp counting: \cite[Proposition 3.8.3]{diamond-shurman}.

\subsubsection*{Column D: Has Complex Multiplication?}

Code:
\begin{lstlisting}
try:
    is_cm = E.has_cm()
    cm_disc = int(E.cm_discriminant()) if is_cm else 0
except Exception:
    is_cm = "error"
    cm_disc = "error"
# Lines 326-330
\end{lstlisting}

References: CM theory: \cite[Chapter II]{silverman:advanced}, \cite{cox:primes}. The 13 CM $j$-invariants: \cite[Appendix A]{silverman:advanced}.

\subsubsection*{Column E: CM Discriminant}
If $E$ does not have CM, this field is $0$.

Code:
\begin{lstlisting}
cm_disc = int(E.cm_discriminant()) if is_cm else 0
# Line 328
\end{lstlisting}

\subsubsection*{Column F: Root Number}
The root number $w \in \{+1, -1\}$.
Code:
\begin{lstlisting}
try:
    root_number = int(E.root_number())
except Exception:
    root_number = "error"
# Lines 333-336
\end{lstlisting}

References: Functional equation: \cite[Chapter 16]{washington:elliptic}. Root number and BSD: \cite{BSD, gross-zagier}. Computational methods: \cite[Section 2.14]{cremona:algorithms}.

\subsection*{ Potential issues}

\begin{enumerate}
\item Rank computation: For curves of large rank ($\geq 2$), \texttt{E.rank()} may be slow or fail. The BSD conjecture is used heuristically for rank 0 and 1. All curves in this dataset ($N \leq 54$) have rank $\leq 1$, so all rank computations are provably correct.

\item Generator computation: For rank 1 curves, \texttt{E.gens()} returns provably correct generators using Cremona's algorithm. Heights are computed to machine precision ($\approx 10^{-15}$ relative error).

\item Galois representations: The method \texttt{non\_surjective()} is provably correct for the primes it returns, but could theoretically miss exceptional primes beyond SageMath's computational limits. For $N \leq 54$, all computations are complete.

\end{enumerate}


\section{Tables}
\begin{table}[H]
\centering
\begin{tabular}{rrrrrrrrr}
1 & 1& 1& 1& 1& 1& 1& 1& 1\\
1 & 2& 2& 2& 2& 2& 2& 2& 2\\
2 & 3& 3& 3& 3& 3& 3& 3& 3\\ 
3 & 0& 0& 0& 0& 0& 0& 0& 0\\
0 & 1& 1& 1& 1& 1& 1& 1& 1\\
1 & 2& 2& 2& 2& 2& 2& 2& 2\\ 
2 & 3& 3& 3& 3& 3& 3& 3& 3\\ 
3 & 0& 0& 0& 0& 0& 0& 0& 0\\
0 & 1& 1& 1& 1& 1& 1& 1& 1\\
1 & 2& 2& 2& 2& 2& 2& 2& 2\\
2 & 3& 3& 3& 3& 3
\end{tabular}
\caption{Residues modulo 4 of abscissas from the
lower envelope of Figure 2.}
\end{table}
\begin{table}[H]
\centering
\begin{tabular}{|l|c|c|c|c|}
\hline
Curve & $c$ & Period & Ratio & Power \\
\hline
$\tau(p_n)$ & 0 & 4.14  & 1.000 & $1.27 \times 10^4$ \\
            &   & 2.06  & 2.010 & $3.98 \times 10^3$ \\
            &   & 3.68  & 1.125 & $1.34 \times 10^3$ \\
            &   & 37.25 & 0.111 & $1.23 \times 10^3$ \\
\hline
$\tau(p_n)$ & 1 & 4.11  & 1.000 & $5.65 \times 10^4$ \\
            &   & 2.06  & 1.995 & $8.60 \times 10^3$ \\
            &   & 3.69  & 1.114 & $5.33 \times 10^3$ \\
\hline
$\tau(p_n)$ & 2 & 4.11  & 1.000 & $5.40 \times 10^4$ \\
            &   & 2.06  & 1.995 & $7.63 \times 10^3$ \\
            &   & 3.69  & 1.114 & $4.62 \times 10^3$ \\
\hline
$\tau(p_n)$ & 3 & 4.11  & 1.000 & $5.61 \times 10^4$ \\
            &   & 2.06  & 1.995 & $8.37 \times 10^3$ \\
            &   & 3.69  & 1.114 & $5.19 \times 10^3$ \\
\hline
\end{tabular}
\caption{Fourier analysis of log-transformed, detrended minimum moduli of eigenvalues of $\chi(|J^{(c)}|)$ for $h(n) = \tau(p_n)$. Data spans $n = 2$ to 400 (399 primes) for fixed-$c$ analyses ($c=0,1,2,3$). For $c=0$, only 298 primes were analyzed. The polynomial secular trend (Table 6) was removed before computing the Fourier transform. The Ratio column shows the ratio of the fundamental period to the period in question.}
\label{tab:fourier_detrended}
\end{table}
\begin{table}[H]
\centering
\begin{tabular}{|l|c|c|c|}
\hline
Curve & Period & Ratio & Power \\
\hline
11a1  & 11.76  & 1.000 & $3.83 \times 10^4$ \\
      & 5.97   & 1.970 & $4.15 \times 10^3$ \\
      & 3.96   & 2.970 & $4.10 \times 10^3$ \\
      & 2.96   & 3.973 & $2.40 \times 10^3$ \\
\hline
14a1  & 5.63   & 1.000 & $3.41 \times 10^4$ \\
      & 2.80   & 2.011 & $4.77 \times 10^3$ \\
      & 2.15   & 2.619 & $3.52 \times 10^3$ \\
\hline
15a1  & 30.77  & 1.000 & $3.92 \times 10^4$ \\
      & 15.38  & 2.000 & $9.26 \times 10^3$ \\
      & 10.26  & 2.999 & $5.06 \times 10^3$ \\
      & 7.69   & 4.001 & $2.78 \times 10^3$ \\
\hline
19a   & 2.84   & 1.000 & $4.09 \times 10^4$ \\
      & 3.39   & 0.838 & $4.04 \times 10^3$ \\
      & 17.39  & 0.163 & $2.12 \times 10^3$ \\
\hline
24a1  & 5.30   & 1.000 & $9.83 \times 10^4$ \\
\hline
37a1  & 6.56   & 1.000 & $3.58 \times 10^4$ \\
      & 3.31   & 1.982 & $4.83 \times 10^3$ \\
      & 2.43   & 2.700 & $4.31 \times 10^3$ \\
\hline
43a1  & 7.30   & 1.000 & $4.49 \times 10^4$ \\
      & 3.65   & 2.000 & $5.34 \times 10^3$ \\
      & 2.43   & 3.004 & $2.61 \times 10^3$ \\
\hline
53a1  & 6.35   & 1.000 & $4.00 \times 10^4$ \\
      & 2.13   & 2.981 & $5.01 \times 10^3$ \\
      & 3.20   & 1.984 & $4.34 \times 10^3$ \\
      & 2.68   & 2.370 & $2.47 \times 10^3$ \\
\hline
\end{tabular}
\caption{Fourier analysis of minimum moduli for elliptic curves using deformed matrices, with $c=1$, for all indices. The Ratio column shows the ratio of the fundamental period to the period in question. Curve 24a1 shows a single dominant peak, or a tight cluster at the edge of FFT resolution with high power. (We show the doubtful
periods in the next table.) The reader will notice that the ratios in rows 15a1 and 43a1 are nearly integers.}
\label{tab:elliptic_fourier_all}
\end{table}
\begin{table}[H]
\centering
\begin{tabular}{|c|c|}
\hline
Period & Power \\
\hline
5.2980 & $9.83 \times 10^4$ \\
5.2805 & $8.31 \times 10^4$ \\
5.3156 & $7.93 \times 10^4$ \\
5.2632 & $4.69 \times 10^4$ \\
5.3333 & $4.05 \times 10^4$ \\
5.2459 & $1.55 \times 10^4$ \\
5.3512 & $9.91 \times 10^3$ \\
5.4054 & $7.60 \times 10^3$ \\
5.3872 & $5.62 \times 10^3$ \\
5.4237 & $3.78 \times 10^3$ \\
5.1948 & $3.04 \times 10^3$ \\
5.4795 & $2.74 \times 10^3$ \\
5.2288 & $2.49 \times 10^3$ \\
5.2117 & $2.03 \times 10^3$ \\
5.4983 & $1.97 \times 10^3$ \\
5.1780 & $1.66 \times 10^3$ \\
5.3691 & $1.35 \times 10^3$ \\
5.5556 & $1.25 \times 10^3$ \\
5.4608 & $1.21 \times 10^3$ \\
5.5749 & $1.18 \times 10^3$ \\
\hline
\end{tabular}
\caption{The uncertain cluster of periods for curve 24a1 at all indices. See `curve 24a1 all indices 21feb26' in \cite{Bre25}.}
\label{tab:curve_24a1_all_periods}
\end{table}
\begin{table}[H]
\centering
\begin{tabular}{|l|c|c|c|c|}
\hline
Curve & Range & Period & Ratio & Power \\
\hline
11a1  & $n < 100$ & 2.62 & 1.000 & $1.27 \times 10^3$ \\
      & $n \geq 100$ & 2.69 & 1.000 & $3.39 \times 10^4$ \\
\hline
14a1  & All & 3.81 & 1.000 & $4.81 \times 10^4$ \\
      &     & 2.11 & 1.806 & $5.92 \times 10^3$ \\
      &     & 4.71 & 0.809 & $3.36 \times 10^3$ \\
\hline
15a1  & All & 5.70 & 1.000 & $4.75 \times 10^4$ \\
      &     & 2.85 & 2.000 & $5.13 \times 10^3$ \\
      &     & 5.78 & 0.986 & $3.30 \times 10^3$ \\
      &     & 2.105 & 2.708 & $3.61 \times 10^3$ \\
\hline
17a1  & All & 3.81 & 1.000 & $4.52 \times 10^4$ \\
      &     & 2.11 & 1.806 & $4.63 \times 10^3$ \\
\hline
\end{tabular}
\caption{Fourier analysis of minimum moduli for elliptic curves using deformed matrices, with $c=1$, for prime indices only. The Ratio column shows the ratio of the fundamental period to the period in question. For curves 11a1 and 15a1, spurious zeros in the first two data points were removed before analysis. Curve 11a1 uniquely exhibits non-stationary behavior with period evolution from 2.62 to 2.69 and a 27-fold power increase at the transition ($n \approx 100$), while other curves show stationary harmonic structure throughout. Curves 14a1 and 17a1 share matching top two periods.}
\label{tab:elliptic_fourier_primes}
\end{table}
\begin{table}[H]
\centering
\begin{tabular}{|l|c|c|c|c|}
\hline
Function & $c$ & Quadratic & Linear & Constant \\
\hline
$\tau(p_n)$ & 0 & $-5.60 \times 10^{-6}$ & $2.02 \times 10^{-3}$ & 2.806 \\
$\tau(p_n)$ & 1 & $-1.06 \times 10^{-6}$ & $5.69 \times 10^{-4}$ & 2.871 \\
$\tau(p_n)$ & 2 & $-2.41 \times 10^{-6}$ & $1.25 \times 10^{-3}$ & 2.820 \\
$\tau(p_n)$ & 3 & $-4.85 \times 10^{-6}$ & $1.84 \times 10^{-3}$ & 2.822 \\
\hline
\end{tabular}
\caption{Secular trends in log-transformed minimum moduli for data from Table 2. The fitted polynomial is $\log(\text{min modulus}) \approx a_2 n^2 + a_1 n + a_0$, where $n$ is the prime index. All cases show small positive linear coefficients, indicating minimum moduli increase slightly with $n$ (moving away from zero). The quadratic terms are negligible.}
\label{tab:secular_trends}
\end{table}
\begin{table}[H]
\centering
\begin{tabular}{|l|c|c|c|}
\hline
Function & Indices & Period & Power \\
\hline
$\tau(p_n)$ & Primes & 4.10 & $2.29 \times 10^4$ \\
\hline
\end{tabular}
\caption{Global (over $c$) Fourier analysis for $h(n) = \tau_p(n)$ and  $n = 2$ to $300, 1 \leq c \leq 2n$, degree of $D_n(c) = n$.)}
\label{tab:poly_c_fourier}
\end{table}
\small
\clearpage
\begin{longtable}{ccccccccc}
\caption{Basic arithmetic invariants for $\Delta$ and elliptic
curves with conductor $\leq 53$.
Column labels: A = Curve label, B = Minimal Weierstrass form
$[a_1,a_2,a_3,a_4,a_6]$, C = Weight, D = Torsion order, E =
Torsion structure, F = Rank. G = Deformed plots with $h(n) = a(n)$
oscillate? H = Deformed plots with $h(n) = a(1 + p_n)$ oscillate? 
K = Deformed plots with $h(n) = a(p_n)$ oscillate? 
The conductor
of $\Delta$ is $1$ and that of each other cusp form is given by in
its label. Thus the conductor of 11a1 is $11$.
\label{tab:csv1}} \\
\toprule
A & B & C & D & E & F & G & H & K\\
\midrule
\endfirsthead
\toprule
A & B & C & D & E & F & G & H & K\\
\midrule
\endhead
\bottomrule
\endfoot
$\Delta$ & --- & 12 & --- & --- & --- &  N & N & Y \\
11a1 & $[0,-1,1,-10,-20]$ & 2 & 5 & $\mathbb{Z}/5\mathbb{Z}$ & 0 &  Y & Y & N \\
14a1 & $[1,0,1,4,-6]$ & 2 & 6 & $\mathbb{Z}/6\mathbb{Z}$ & 0 &  Y & Y & N\\
15a1 & $[1,1,1,-10,-10]$ & 2 & 8 & $\mathbb{Z}/2\mathbb{Z} \times \mathbb{Z}/4\mathbb{Z}$ & 0  &  Y & Y\\
17a1 & $[1,-1,1,-1,-14]$ & 2 & 4 & $\mathbb{Z}/4\mathbb{Z}$ & 0  &  N & Y\\
19a1 & $[0,1,1,0,-9]$ & 2 & 3 & $\mathbb{Z}/3\mathbb{Z}$ & 0  &  Y & N\\
20a1 & $[0,1,0,4,4]$ & 2 & 6 & $\mathbb{Z}/6\mathbb{Z}$ & 0 &  N & N\\
21a1 & $[1,0,0,-1,-2]$ & 2 & 8 & $\mathbb{Z}/2\mathbb{Z} \times \mathbb{Z}/4\mathbb{Z}$ & 0 & N & Y\\
24a1 & $[0,-1,0,-4,4]$ & 2 & 8 & $\mathbb{Z}/2\mathbb{Z} \times \mathbb{Z}/4\mathbb{Z}$ & 0 &  Y & N\\
26a1 & $[1,-1,1,-3,3]$ & 2 & 3 & $\mathbb{Z}/3\mathbb{Z}$ & 0 & N & Y\\
26b1 & $[1,-1,1,-56,-10]$ & 2 & 7 & $\mathbb{Z}/7\mathbb{Z}$ & 0 & N  & N\\
27a1\footnote{Let $h_0 = 1$ and $h_n$ be the coefficient
of $q_*^{p_n}$ in the Fourier expansion of curve 27a1 for $n$
positive. Let $j_0 = 1$ and let $j_n$ for $n$ positive
satisfy $j_n = n  h_n - \sum_{r=0}^{n-1} j_r h_{n-r}$.
For $n \leq 400$, we observed that
$j_n = -2 \times 2^{n/2}$ for even positive $n$, and $j_n = 0$ for odd
positive $n$. Obviously, we can reverse this and
write down a conjecture for the Fourier expansion.
See \texttt{curve 27a1 on primes 1mar26} in \cite{Bre25}.
(Sadly,
the $j_n$ for $h_n = \tau(p_n)$ do not appear to have similar
behavior.)
}
& $[0,0,1,0,-7]$ & 2 & 3 & $\mathbb{Z}/3\mathbb{Z}$ & 0 & N& Y\footnote{More precisely, it
appears to us that the minima of 27a1b on primes lie on the sum of a
logarithmic curve and an oscillating one, or, the sum of at least two
oscillatory curves, one
having a much greater wavelength than the other. Fourier analysis by
Anthropic's ``Claude'' model favors the first possibility.
See \texttt{curve 27a1 on primes 2mar26}
in \cite{Bre25}.}\\
30a1 & $[1,-1,0,-3,2]$ & 2 & 6 & $\mathbb{Z}/6\mathbb{Z}$ & 0 \\
32a1 & $[0,0,0,4,0]$ & 2 & 4 & $\mathbb{Z}/4\mathbb{Z}$ & 0 \\
33a1 & $[1,1,0,-2,0]$ & 2 & 4 & $\mathbb{Z}/2\mathbb{Z} \times \mathbb{Z}/2\mathbb{Z}$ & 0 \\
34a1 & $[1,1,0,-1,-1]$ & 2 & 6 & $\mathbb{Z}/6\mathbb{Z}$ & 0 \\
35a1 & $[0,1,1,1,-2]$ & 2 & 3 & $\mathbb{Z}/3\mathbb{Z}$ & 0 \\
36a1 & $[0,0,0,-3,-2]$ & 2 & 6 & $\mathbb{Z}/6\mathbb{Z}$ & 0 \\
37a1 & $[0,0,1,-1,0]$ & 2 & 1 & 0 & 1 & Y & Y\\
37b1 & $[0,1,1,-23,-50]$ & 2 & 3 & $\mathbb{Z}/3\mathbb{Z}$ & 0 \\
38a1 & $[1,-1,1,-1,0]$ & 2 & 3 & $\mathbb{Z}/3\mathbb{Z}$ & 0 \\
38b1 & $[1,1,1,-4,4]$ & 2 & 5 & $\mathbb{Z}/5\mathbb{Z}$ & 0 \\
39a1 & $[1,-1,0,0,-2]$ & 2 & 4 & $\mathbb{Z}/2\mathbb{Z} \times \mathbb{Z}/2\mathbb{Z}$ & 0 \\
40a1 & $[0,0,0,-3,2]$ & 2 & 4 & $\mathbb{Z}/2\mathbb{Z} \times \mathbb{Z}/2\mathbb{Z}$ & 0 \\
42a1 & $[1,1,0,-2,-1]$ & 2 & 8 & $\mathbb{Z}/8\mathbb{Z}$ & 0 \\
43a1 & $[0,1,1,0,0]$ & 2 & 1 & 0 & 1 & Y & Y\\
44a1 & $[0,-1,0,-5,4]$ & 2 & 3 & $\mathbb{Z}/3\mathbb{Z}$ & 0 \\
45a1 & $[1,1,1,-1,-1]$ & 2 & 2 & $\mathbb{Z}/2\mathbb{Z}$ & 0 \\
46a1 & $[1,0,1,-1,0]$ & 2 & 2 & $\mathbb{Z}/2\mathbb{Z}$ & 0 \\
48a1 & $[0,0,0,1,0]$ & 2 & 4 & $\mathbb{Z}/2\mathbb{Z} \times \mathbb{Z}/2\mathbb{Z}$ & 0 \\
49a1 & $[1,-1,0,-2,0]$ & 2 & 2 & $\mathbb{Z}/2\mathbb{Z}$ & 0 \\
50a1 & $[1,1,1,-2,-1]$ & 2 & 3 & $\mathbb{Z}/3\mathbb{Z}$ & 0 \\
50b1 & $[1,1,1,0,0]$ & 2 & 5 & $\mathbb{Z}/5\mathbb{Z}$ & 0 \\
51a1 & $[1,1,0,0,-1]$ & 2 & 3 & $\mathbb{Z}/3\mathbb{Z}$ & 0 \\
52a1 & $[0,0,0,1,-2]$ & 2 & 2 & $\mathbb{Z}/2\mathbb{Z}$ & 0 \\
53a1 & $[1,-1,1,-8,6]$ & 2 & 1 & 0 & 1 & Y & N\\
\end{longtable}
\clearpage
\begin{longtable}{ccccccccc}
\caption{Mordell-Weil generators and Galois representation surjectivity
for $\Delta$ and elliptic curves with conductor $\leq 53$.
Column labels: A = Curve label, B = $j$-invariant, C = Generator canonical heights
(approx.), D = Regulator (approx.), E = Galois rep. surjective for all
primes?, F = Deformed plots with $h(n) = a(n)$ oscillate? G = Deformed plots
with $h(n) = a(1 + p_n)$ oscillate? H = Deformed plots with $h(n) = a(p_n)$ oscillate? 
Heights shown for rank 1 curves only. For rank 0 curves: regulator = 1.
\label{tab:csv2}} \\
\toprule
A & B & C & D & E & F & G & H\\
\midrule
\endfirsthead
\toprule
A & B & C & D & E & F & G & H\\
\midrule
\endhead
\bottomrule
\endfoot
$\Delta$ & --- & --- & --- & Y & N & N & Y\\
11a1 & $-2^{12}$ & none & 1 & N & Y & Y & N\\
14a1 & $-3^3 \cdot 5^3/2^2$ & none & 1 & N & Y & Y & N\\
15a1 & $-5^2 \cdot 241^3 / (2^5 \cdot 3^3)$ & none & 1 & N & Y & Y\\
17a1 & $-17 \cdot 373^3/2^{17}$ & none & 1 & N & N & Y\\
19a1 & $-2^{15}/19$ & none & 1 & N & Y & N\\
20a1 & $2^{18}/5$ & none & 1 & N & N & N\\
21a1 & $-3^3\cdot 5^3 \cdot 101^3/(7\cdot 2^{21})$ & none & 1 & N & N & Y\\
24a1 & $2^{12}/3$ & none & 1 & N & Y & N\\
26a1 & $-17^3 \cdot 41^3 / (2 \cdot 13^3)$ & none & 1 & N & N & Y\\
26b1 & $-7\cdot 11^3/13$ & none & 1 & N & N & N\\
27a1 & $0 \footnote{\it{sic}. \rm The $j$-invariant is zero because 27a1 has complex multiplication by
$\mathbb{Z}[e^{2 \pi i/3}]$.  See Theorem ~III.10.1 for the classification
$|\mathrm{Aut}(E)| \in \{2, 4, 6\}$
and its relation to $j(E) \in \{0, 1728\}$ in \cite{Si} and 
Theorem~5.21 (p.~23) for the proof that
$\mathrm{Aut}(E) \cong \mathbb{Z}/2\mathbb{Z}$,
$\mathbb{Z}/4\mathbb{Z}$, or $\mathbb{Z}/6\mathbb{Z}$
according as $j(E) \neq 0, 1728$; $j(E) = 1728$; or $j(E) = 0$ in \cite{F}. An obvious thought is that the value zero is connected to the footnotes in the row for curve 27a1 in the previous table.}
$ & none & 1 & N & N\\
30a1 & $-2^4 \cdot 5^3 \cdot 7^3 / (3 \cdot 31)$ & none & 1 & N \\
32a1 & $1728$ & none & 1 & N \\
33a1 & $2^5 \cdot 5^3 / 3$ & none & 1 & N \\
34a1 & $-2^4 \cdot 17^3 / 3$ & none & 1 & N \\
35a1 & $-5 \cdot 29^3 / 7$ & none & 1 & N \\
36a1 & $2^4 \cdot 3^3 \cdot 5^3$ & none & 1 & N \\
37a1 & $-2^{15} \cdot 3^3 / 37$ & $0.0511$ & $0.0511$ & Y & Y & Y\\
37b1 & $-9317^3/37^3$ & none & 1 & N \\
38a1 & $-2^4 \cdot 19^3 / 3$ & none & 1 & N \\
38b1 & $-5^2 \cdot 127^3 / (2 \cdot 19)$ & none & 1 & N \\
39a1 & $-2^4 \cdot 5^3 \cdot 11^3 / (3 \cdot 13)$ & none & 1 & N \\
40a1 & $2^{10} \cdot 5^3 / 3$ & none & 1 & N \\
42a1 & $-2^4 \cdot 3^3 \cdot 5^3 \cdot 23^3 / (7 \cdot 43)$ & none & 1 & N \\
43a1 & $-2^{12} \cdot 3^3 / 43$ & $0.2043$ & $0.2043$ & N & Y & Y\\
44a1 & $2^4 \cdot 5^3 / 11$ & none & 1 & N \\
45a1 & $-5^3 \cdot 3^3 \cdot 2^4 / (5\cdot 3)$ & none & 1 & N \\
46a1 & $-2^4 \cdot 23^3 / 3$ & none & 1 & N \\
48a1 & $2^8 \cdot 3^3$ & none & 1 & N \\
49a1 & $-7^5$ & none & 1 & N \\
50a1 & $-2^4 \cdot 5^3 \cdot 7^3 / (5 \cdot 3)$ & none & 1 & N \\
50b1 & $2^4 \cdot 5^3 \cdot 3^3 / 25$ & none & 1 & N \\
51a1 & $-2^4 \cdot 5^3 \cdot 17^3 / (3 \cdot 17)$ & none & 1 & N \\
52a1 & $2^6 \cdot 3^3 \cdot 5^3 / 13$ & none & 1 & N \\
53a1 & $-2^{18} / 53$ & $0.2581$ & $0.2581$ & N & Y & N\\
\end{longtable}

\clearpage

\begin{longtable}{cccccc}
\caption{Galois representation exceptional primes and modular group data for $\Delta$ and elliptic curves with conductor $\leq 53$. 
Column labels: A = Curve label, B = Galois image type (S = surjective, NS =non-surjective, CM = has complex multiplication), C = Index of $\Gamma_0(N)$ in $\text{SL}_2(\mathbb{Z})$, D = Deformed plots with $h(n)
= a(n)$ oscillate? E = Deformed plots with $h(n) = a(1 + p_n)$ oscillate?
F = Deformed plots with $h(n) = a(p_n)$ oscillate? 
\label{tab:csv3}} \\
\toprule
A & B & C & D & E & F\\
\midrule
\endfirsthead
\toprule
A & B & C & D & E & F\\
\midrule
\endhead
\bottomrule
\endfoot
$\Delta$ & surj. (wt 12) & 1 & N & N & Y\\
11a1 & NS [5] & 12 & Y & Y & N\\
14a1 & NS [2, 3] & 24 & Y & Y & N\\
15a1 & NS [2, 5] & 24 & Y & Y\\
17a1 & NS [2] & 24 & N & Y\\
19a1 & NS [3] & 24 & Y & N\\
20a1 & NS [2, 3] & 48 &  N & N\\
21a1 & NS [2, 7] & 48 & N & Y\\
24a1 & NS [2, 5] & 48 & Y & N\\
26a1 & NS [3, 13] & 48 & N & Y \\
26b1 & NS [3, 7] & 48 & N & N\\
27a1 & CM & 36 & N\\
30a1 & NS [2, 3, 5] & 72 \\
32a1 & CM & 48 \\
33a1 & NS [2, 11] & 72 \\
34a1 & NS [2, 3] & 72 \\
35a1 & NS [3, 7] & 72 \\
36a1 & CM & 72 \\
37a1 & NS [37] & 48 & Y &  Y\\
37b1 & NS [3] & 48 \\
38a1 & NS [3, 19] & 72 \\
38b1 & NS [3, 5] & 72 \\
39a1 & NS [2, 13] & 96 \\
40a1 & NS [2, 5] & 96 \\
42a1 & NS [2, 3, 7] & 144 \\
43a1 & NS [43] & 56 & Y & Y\\
44a1 & NS [2, 3] & 96 \\
45a1 & NS [2, 3] & 108 \\
46a1 & NS [2, 23] & 96 \\
48a1 & NS [2, 5] & 96 \\
49a1 & CM & 84 \\
50a1 & NS [3, 5] & 120 \\
50b1 & NS [3, 5] & 120 \\
51a1 & NS [2, 3] & 96 \\
52a1 & NS [2, 13] & 144 \\
53a1 & NS [53] & 72 & Y & N \\
\end{longtable}

\clearpage

\begin{longtable}{ccccccccc}
\caption{Modular curve genus, complex multiplication, and functional equation data for $\Delta$ and elliptic curves with conductor $\leq 53$. 
Column labels: A = Curve label, B = Genus of $X_0(N)$, C = Number of cusps of $X_0(N)$, D = Complex multiplication?, E = CM discriminant (if applicable), F = Root number of L-function, G = Deformed plots with $h(n) = a(n)$ oscillate? H = Deformed plots with $h(n) = a(1 + p_n)$ oscillate?
K = Deformed plots with $h(n) = a(p_n)$ oscillate? 
\label{tab:csv4}} \\
\toprule
A & B & C & D & E & F & G & H & K\\
\midrule
\endfirsthead
\toprule
A & B & C & D & E & F & G & H & K\\
\midrule
\endhead
\bottomrule
\endfoot
$\Delta$ & 0 & 1 & N & 0 & +1 & N & N & Y\\
11a1 & 1 & 2 & N & 0 & +1 & Y & Y & N\\
14a1 & 1 & 4 & N & 0 & +1 & Y & Y & N\\
15a1 & 1 & 4 & N & 0 & +1 & Y & Y\\
17a1 & 1 & 4 & N & 0 & +1 & N & Y\\
19a1 & 1 & 3 & N & 0 & +1 & Y & N\\
20a1 & 1 & 6 & N & 0 & +1 & N & N\\
21a1 & 1 & 8 & N & 0 & +1 & N & Y\\
24a1 & 1 & 8 & N & 0 & +1 & Y & N\\
26a1 & 2 & 6 & N & 0 & +1 & N & Y\\
26b1 & 2 & 6 & N & 0 & +1 & N & N\\
27a1 & 1 & 4 & Y & -3 & +1 & N\\
30a1 & 3 & 12 & N & 0 & +1 \\
32a1 & 1 & 6 & Y & -4 & +1 \\
33a1 & 3 & 12 & N & 0 & +1 \\
34a1 & 3 & 9 & N & 0 & +1 \\
35a1 & 3 & 12 & N & 0 & +1 \\
36a1 & 3 & 12 & Y & -3 & +1 \\
37a1 & 2 & 4 & N & 0 & -1 & Y & Y\\
37b1 & 2 & 4 & N & 0 & +1 \\
38a1 & 3 & 9 & N & 0 & +1 \\
38b1 & 3 & 9 & N & 0 & +1 \\
39a1 & 3 & 16 & N & 0 & +1 \\
40a1 & 3 & 12 & N & 0 & +1 \\
42a1 & 5 & 24 & N & 0 & +1 \\
43a1 & 3 & 4 & N & 0 & -1 & Y & Y\\
44a1 & 3 & 12 & N & 0 & +1 \\
45a1 & 5 & 12 & N & 0 & +1 \\
46a1 & 5 & 12 & N & 0 & +1 \\
48a1 & 5 & 16 & N & 0 & +1 \\
49a1 & 3 & 8 & Y & -7 & +1 \\
50a1 & 5 & 15 & N & 0 & +1 \\
50b1 & 5 & 15 & N & 0 & +1 \\
51a1 & 5 & 12 & N & 0 & +1 \\
52a1 & 5 & 18 & N & 0 & +1 \\
53a1 & 5 & 6 & N & 0 & -1 & Y & N\\
\end{longtable}
\begin{table}[h]
\centering
\begin{tabular}{ll}
\hline
\textbf{Figure} & \textbf{Jupyter SageMath notebook} \\
\hline
Figure 1  & \texttt{tauprime=h logarithmic envelopes prec100 n400 1dec25} \\
Figure 2  & \texttt{tauprime c=2 4feb26} \\
Figure 3  & \texttt{tauprime c=2 4feb26} \\
Figure 4  & \texttt{tauprime c=3 4feb26} \\
Figure 5  & \texttt{tauprime many c's 18feb26} \\
Figure 6  & \texttt{tauprime controls 22feb26} \\
Figure 7  & \texttt{undeformed primeTau 5oct25} \\
Figure 8A  & \texttt{curve 11a1 29jan26} \\
Figure 8B  & \texttt{curve 11a1 $h(n)=a(p_n)$ 9mar26
} \\
Figure 8C  & \texttt{curve 11a1 $h(n)=a(p_n+1)$ 8mar26
} \\
Figure 9A & \texttt{curve 14a1 9feb26} \\
Figure 9B  & \texttt{curve 14a1 $h(n)=a(p_n)$. 9mar26
} \\
Figure 9C  & \texttt{curve 14a1 $h(n)=a(p_n+1)$. 9mar26
} \\
Figure 10A & \texttt{curve 15a1 10feb26} \\
Figure 10B & \texttt{curve 15a1 on primes 10feb26} \\
Figure 11A & \texttt{curve 17a1 30jan26} \\
Figure 11B & \texttt{curve 17a1 on primes 2feb26} \\
Figure 12A & \texttt{curve 19a1 31jan26} \\
Figure 12B & \texttt{curve 19a1 on primes 21feb26} \\
Figure 13A & \texttt{curve 20a1 11feb26} \\
Figure 13B & \texttt{curve 20a1 on primes 10feb26} \\
Figure 14A & \texttt{curve 21a1 11feb26} \\
Figure 14B & \texttt{curve 21a1 on primes 27feb26}\\
Figure 15A & \texttt{curve\_24a1 all indices 21feb26} \\
Figure 15B & \texttt{curve 24a1 on primes 23feb26} \\
Figure 16A & \texttt{curve 26a1 28feb26}\\
Figure 16B & \texttt{curve 26a1 on primes 28feb26} \\
Figure 17A & \texttt{curve 26b1 28feb26}\\
Figure 17B & \texttt{curve 26b1 on primes} \\
Figure 18A & \texttt{crv 27a1 28feb26} \\
Figure 18B & \texttt{curve 27a1 on primes 1mar26no2}\\
Figure 19A & \texttt{curve 37a1 1feb26} \\
Figure 19B &  \texttt{curve 37a1 on primes 23feb26} \\
Figure 20A & \texttt{curve 43a1no2 22feb26} \\
Figure 20B & \texttt{curve 43a1 on primes 28feb26} \\
Figure 21A & \texttt{crv 53a1 2feb26} \\
Figure 21B & \texttt{curve 53a1 on primes 28feb26} \\
Figure 22 & \texttt{deformed tau 18oct25} \\
\hline
\end{tabular}
\caption{Figure references. The notebooks are in our repository \cite{Bre25}.}
\label{tab:figure-references}
\end{table}

\section{Figures}
\begin{figure}[H]
    \centering
    \includegraphics[width=1\textwidth]{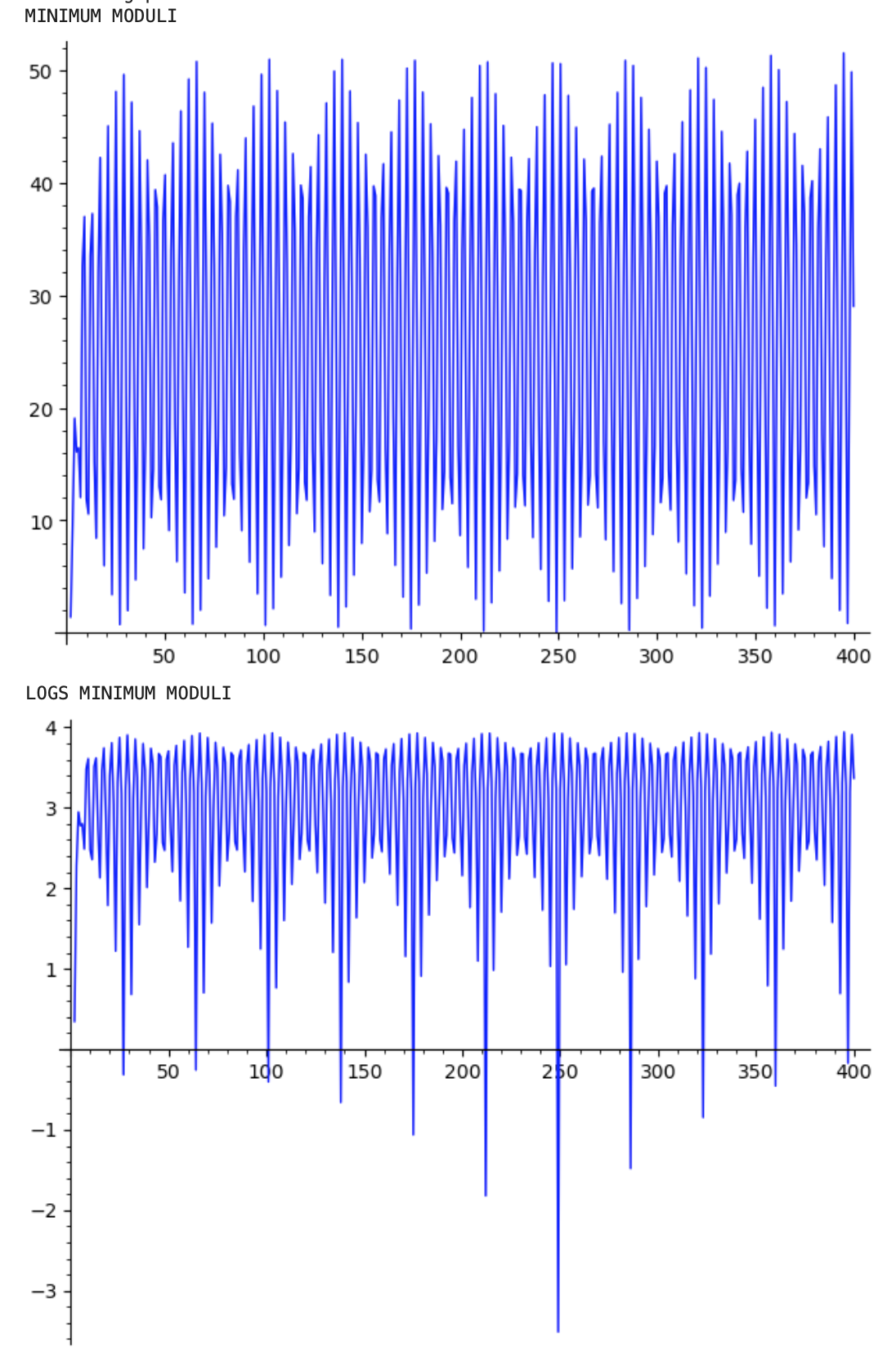}
    \caption{Minimum moduli for $h(n)=\tau(p_n)$ with $c=1$.}
    \label{fig:minima}
\end{figure}
\begin{figure}[H]
    \centering
    \includegraphics[width=1\textwidth]{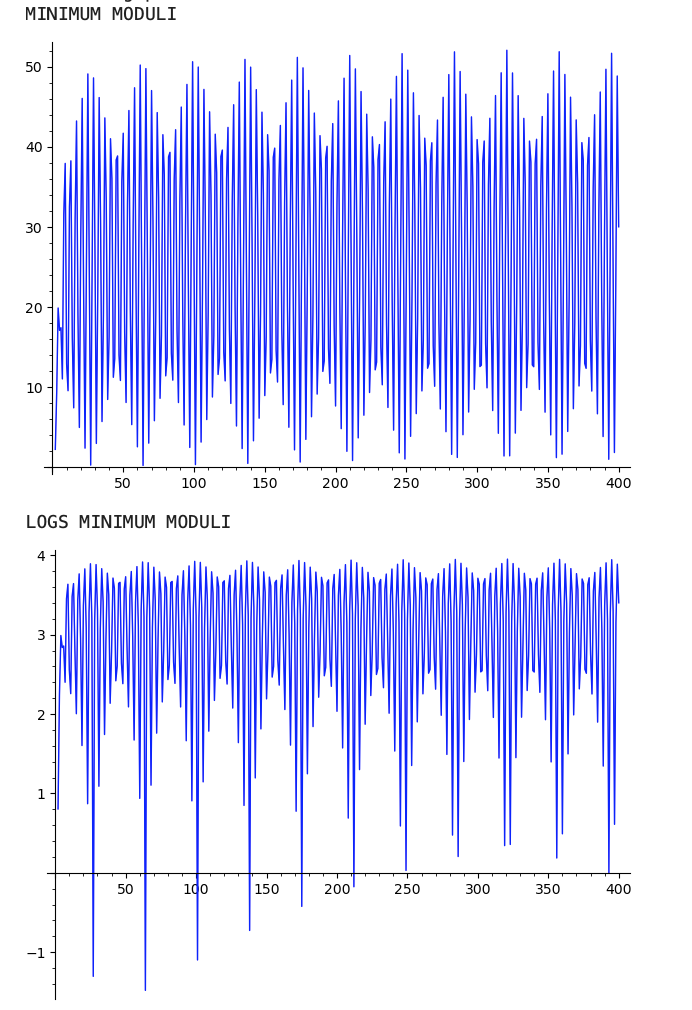}
    \caption{Minimum moduli and their logarithms for $h(n) = \tau(p_n)$
   with $c = 2$.} 
    \label{fig:tauprime_c_2}
\end{figure}
\begin{figure}[H]
    \centering
    \includegraphics[width=1\textwidth]{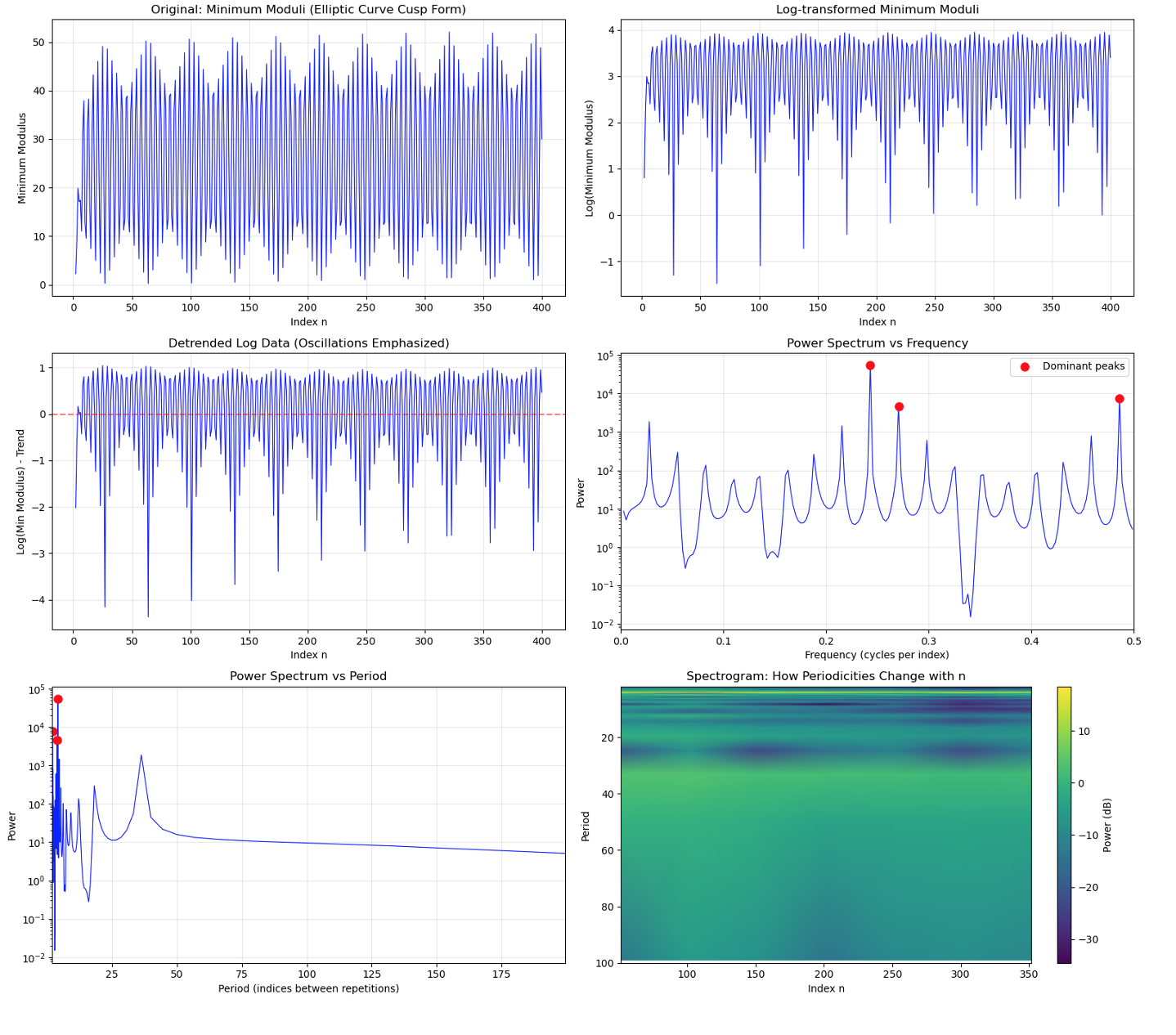}
    \caption{Before and after de-trending
    for $h(n) = \tau(p_n)$ with $c = 2$.}  
    \label{fig:trending_vs_non}
\end{figure}
\begin{figure}[H]
    \centering
    \includegraphics[width=1\textwidth]{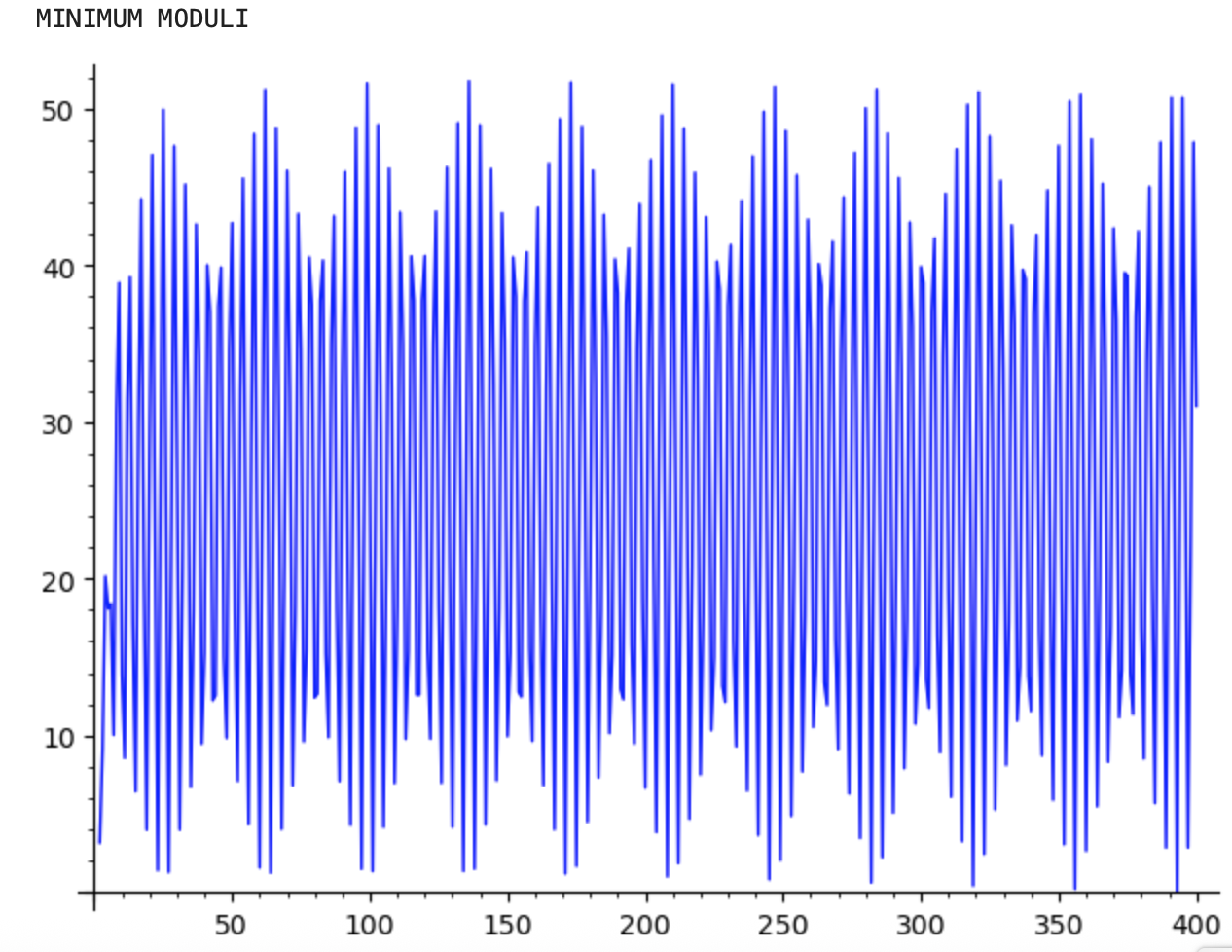}
    \caption{Minimum moduli and their logarithms for $h(n) = \tau(p_n)$
   with $c = 3$.} 
    \label{fig:tauprime_c_3}
\end{figure}
\begin{figure}[H]
    \centering
    \includegraphics[width=1\textwidth]{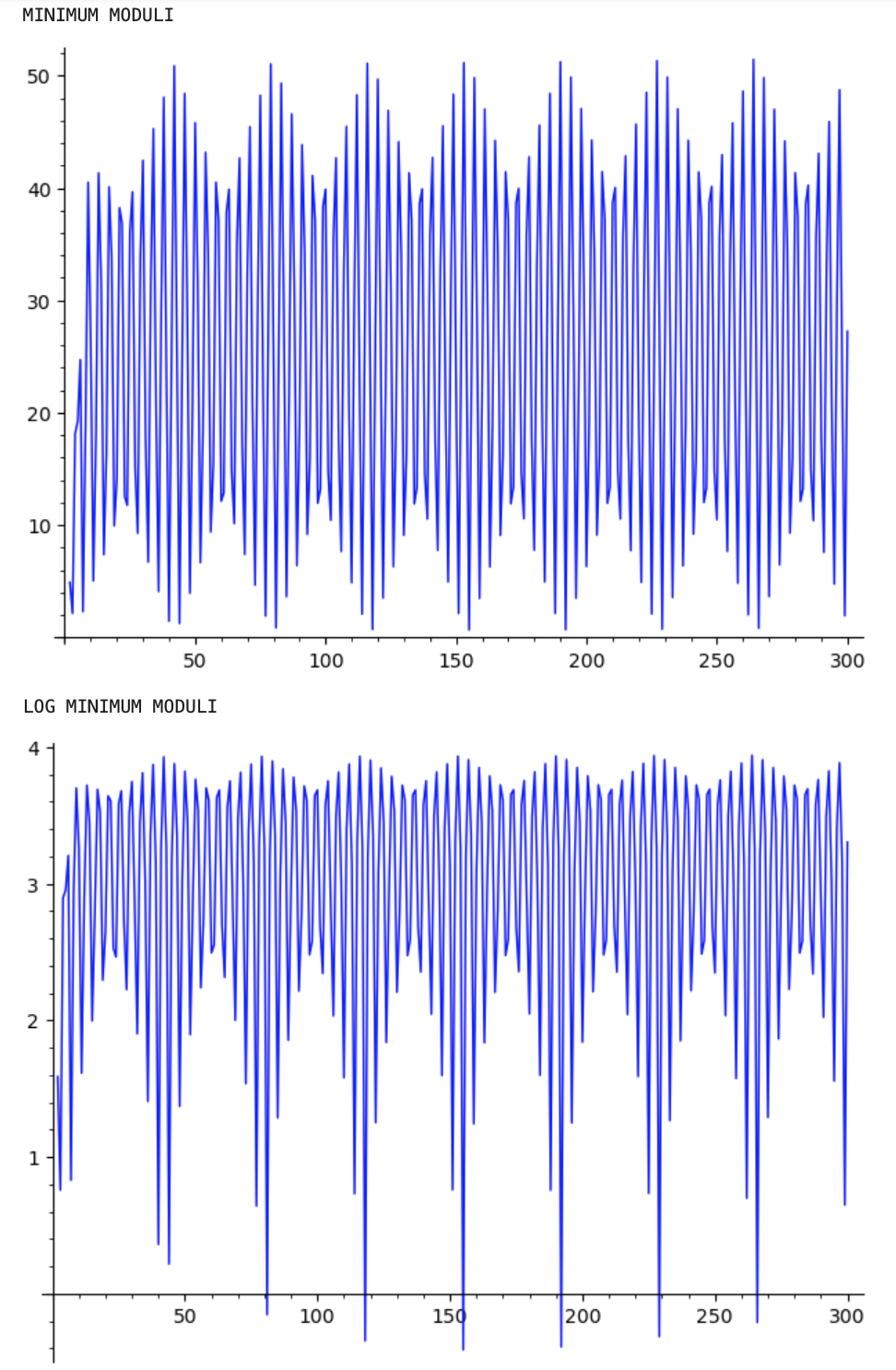}
    \caption{Minimum moduli over many values of $c$ for $h(n) = \tau_p(n)$.}
    \label{fig:tauprime_many_c}
\end{figure}
\begin{figure}[H]
    \centering
    \includegraphics[width=1\textwidth]{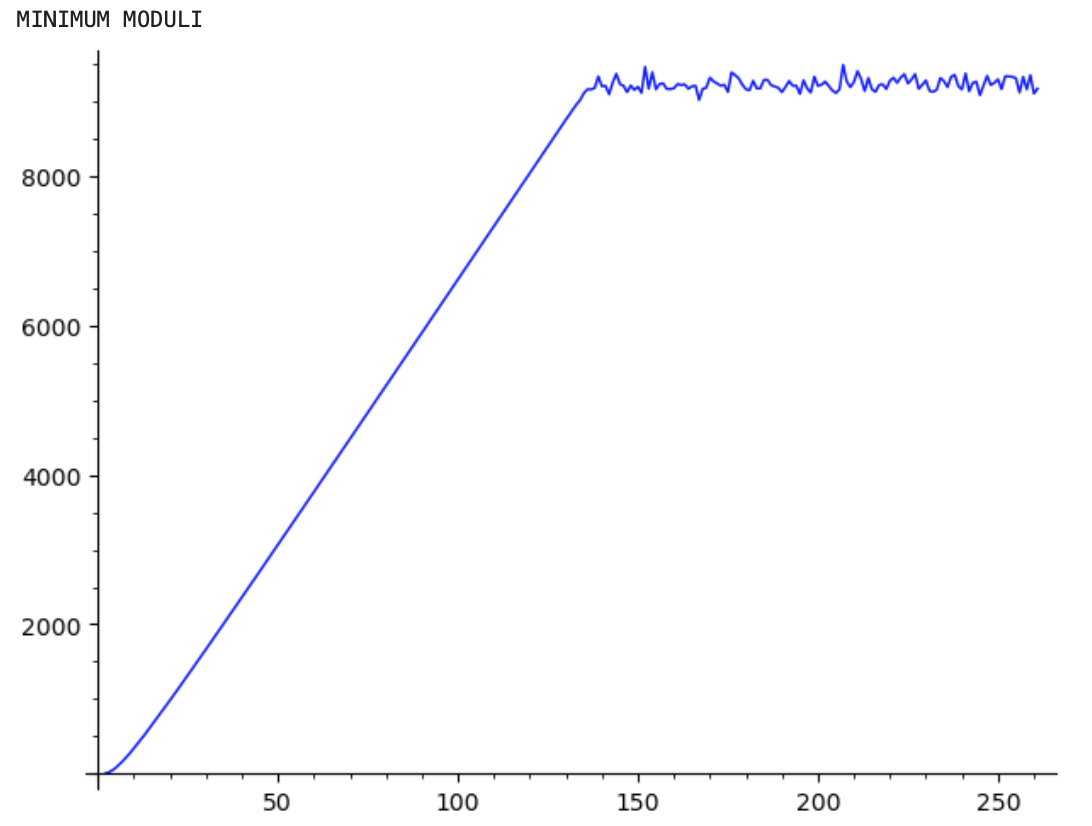}
    \caption{Experimental control: minimum moduli for $h(n) = \tau(p_n+ 1)$ with $c = 1$.} 
    \label{fig:tauprime_control}
\end{figure}
\begin{figure}[H]
    \centering
    \includegraphics[width=1\textwidth]{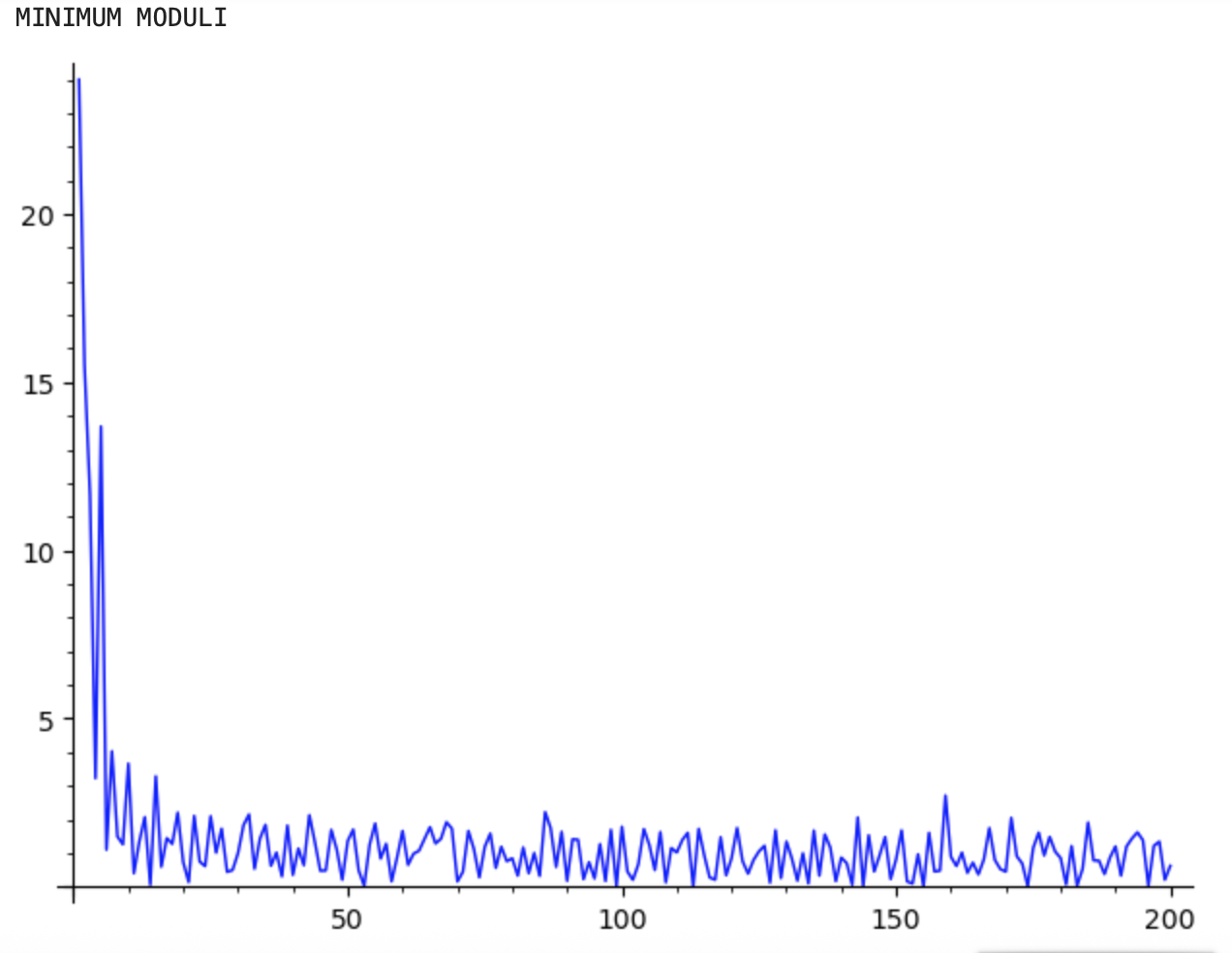}
    \caption{Minimum moduli for $h(n) = \tau(p_n)$ before deformation.}
    \label{fig:undeformed_primeTau_min_moduli}
\end{figure}
\begin{figure}[H]
    \centering
    \begin{subfigure}[t]{0.48\textwidth}
        \centering
        \includegraphics[width=\textwidth]{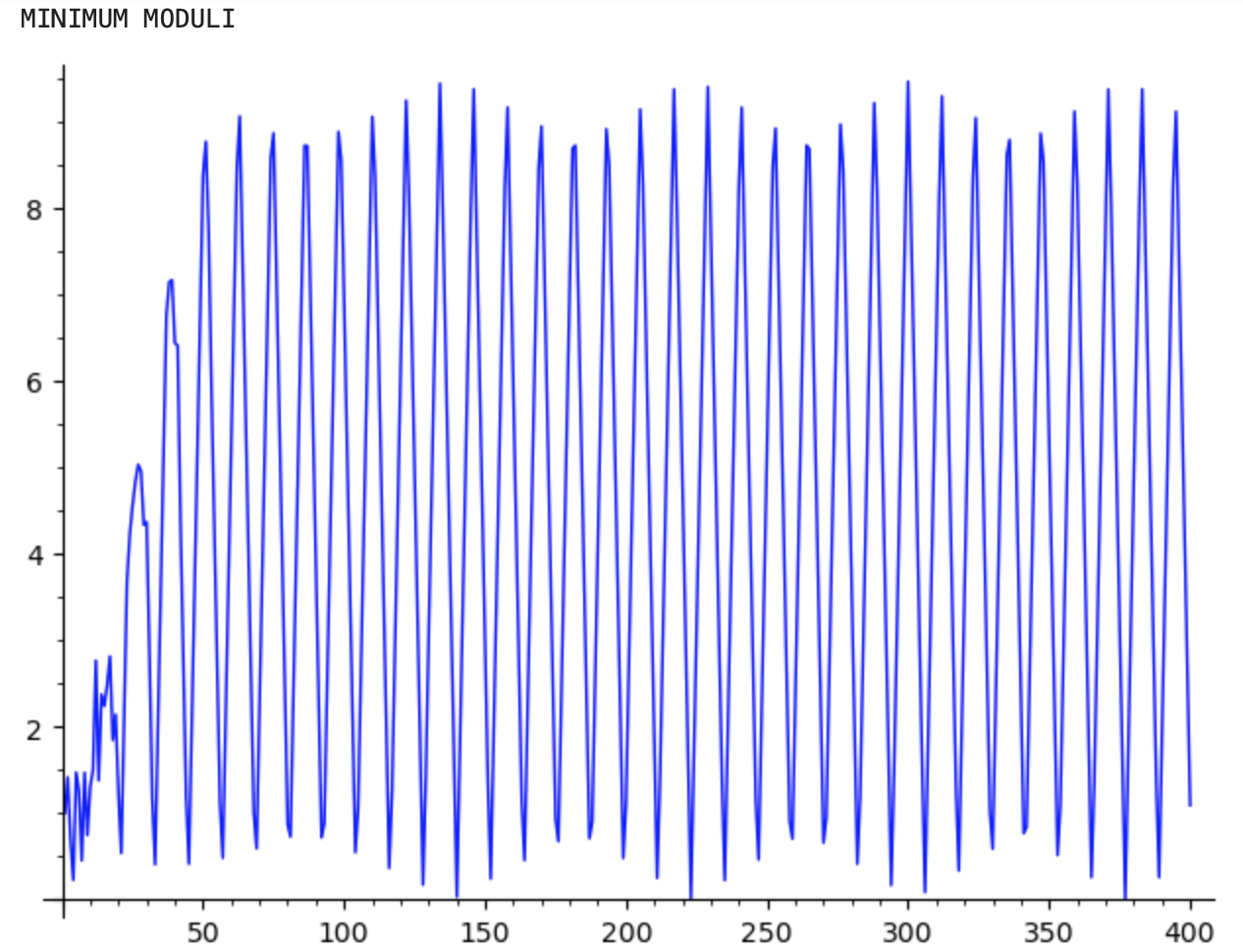}
        \caption{$h(n)=a_n$.}
        \label{fig:curve_11a1_all}
    \end{subfigure}
     \begin{subfigure}[t]{0.48\textwidth}
        \centering
        \includegraphics[width=\textwidth]{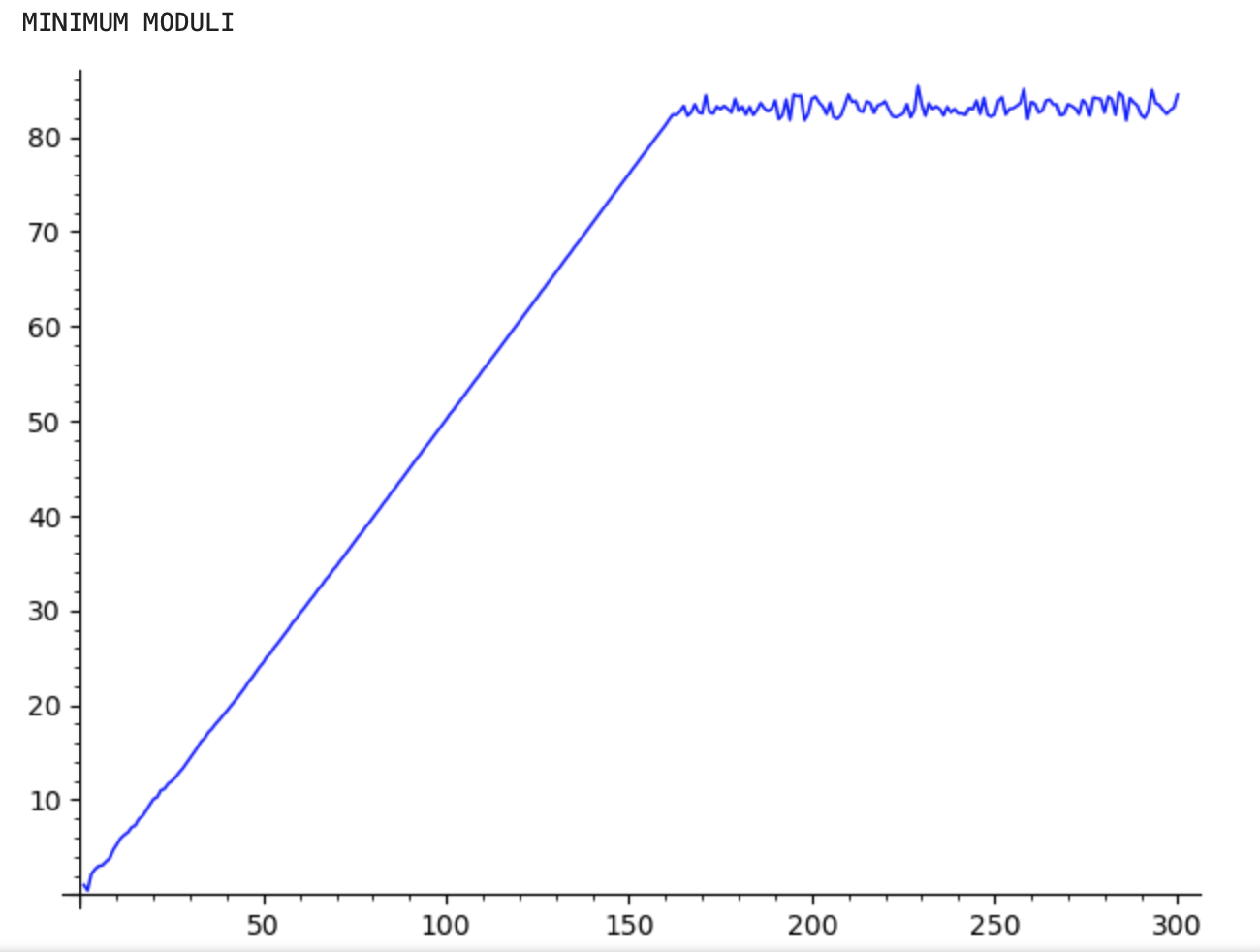}
        \caption{$h(n) = a_{p_n}$.}
        \label{fig:curve11a1_primes}
    \end{subfigure}
    \hfill
    \begin{subfigure}[t]{0.48\textwidth}
        \centering
        \includegraphics[width=\textwidth]{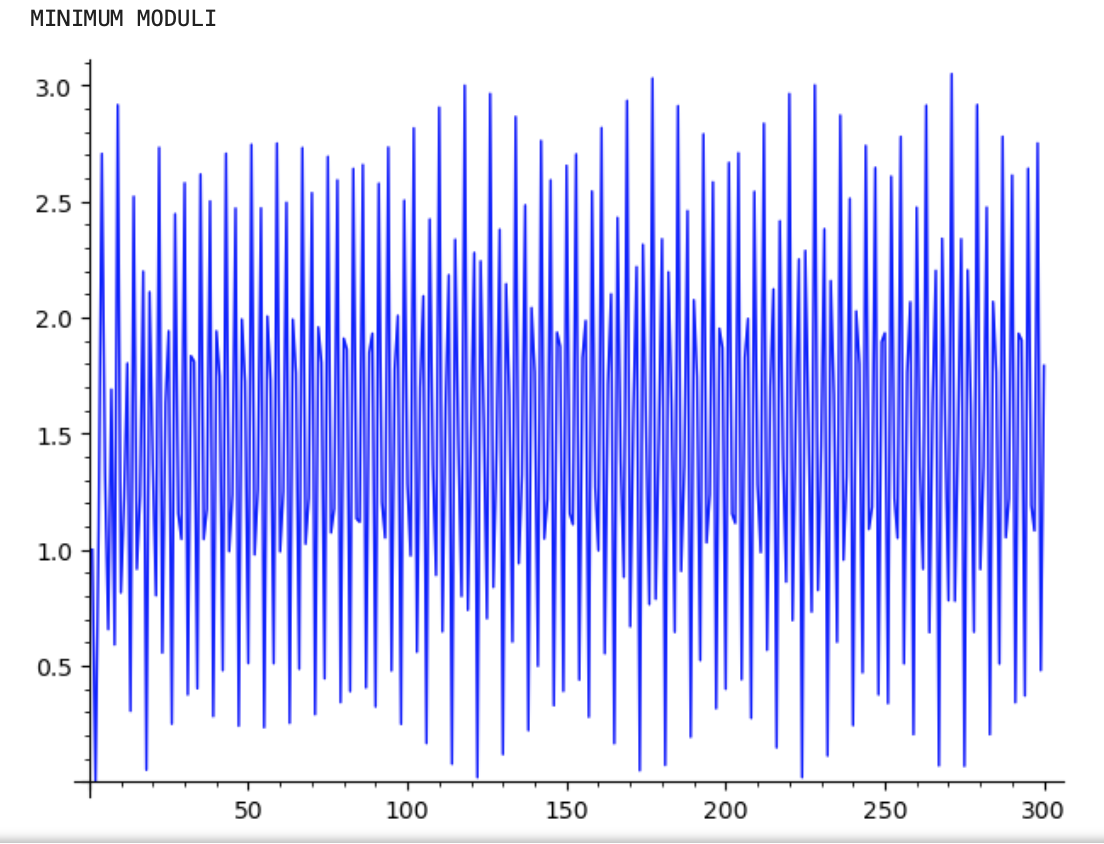}
        \caption{$h(n) = a_{1+p_n}$.}
        \label{fig:curve11a1_shift_minus_1}
    \end{subfigure}
    \caption{Curve 11a1 with $c = 1$. }
    \label{fig:comparison_crv11a1}
\end{figure}
\begin{figure}[H]
    \centering
    \begin{subfigure}[t]{0.48\textwidth}
        \centering
        \includegraphics[width=\textwidth]{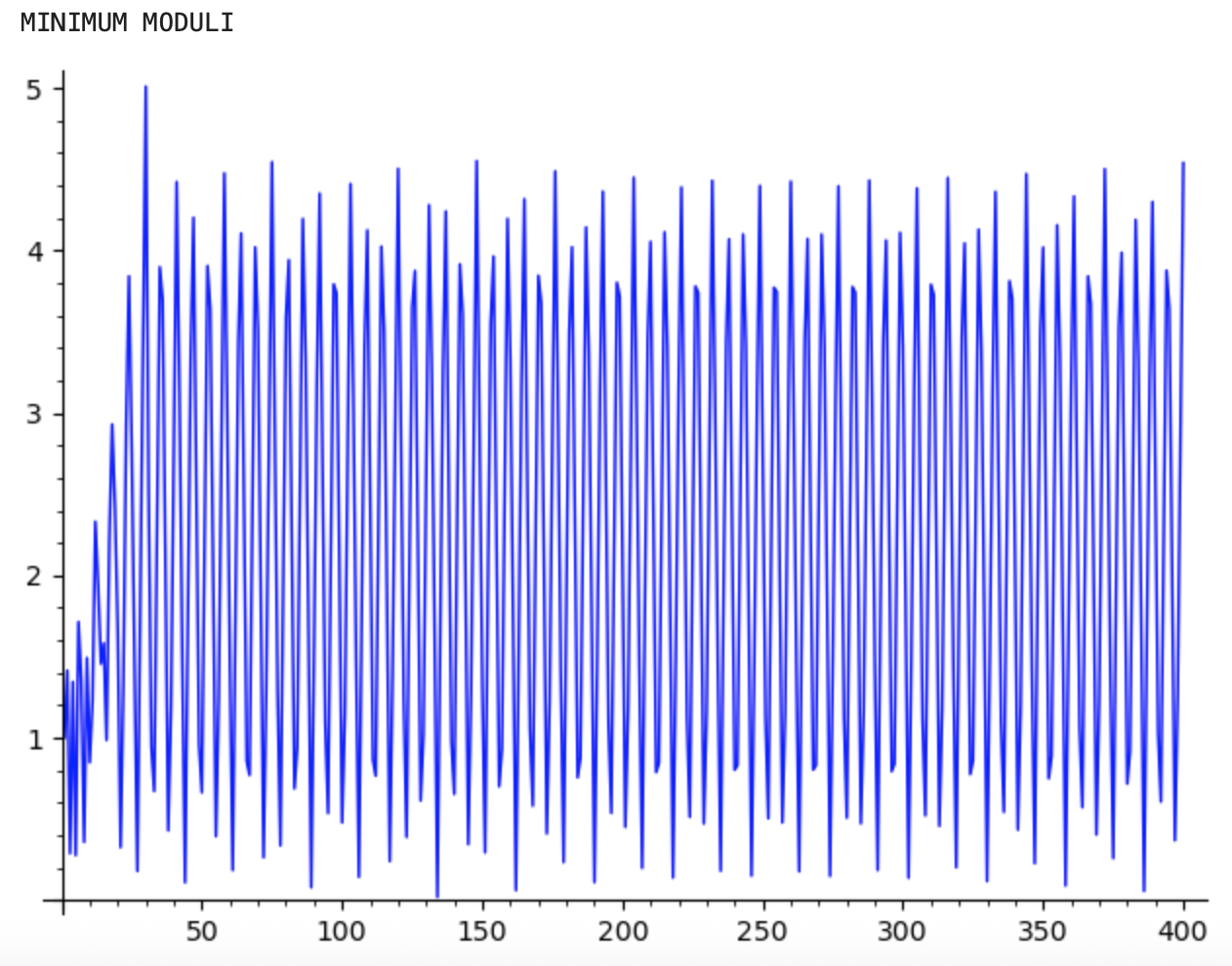}
        \caption{$h(n)=a_n$.}
        \label{fig:curve_14a1}
    \end{subfigure}
    \hfill
    \begin{subfigure}[t]{0.48\textwidth}
        \centering
        \includegraphics[width=\textwidth]{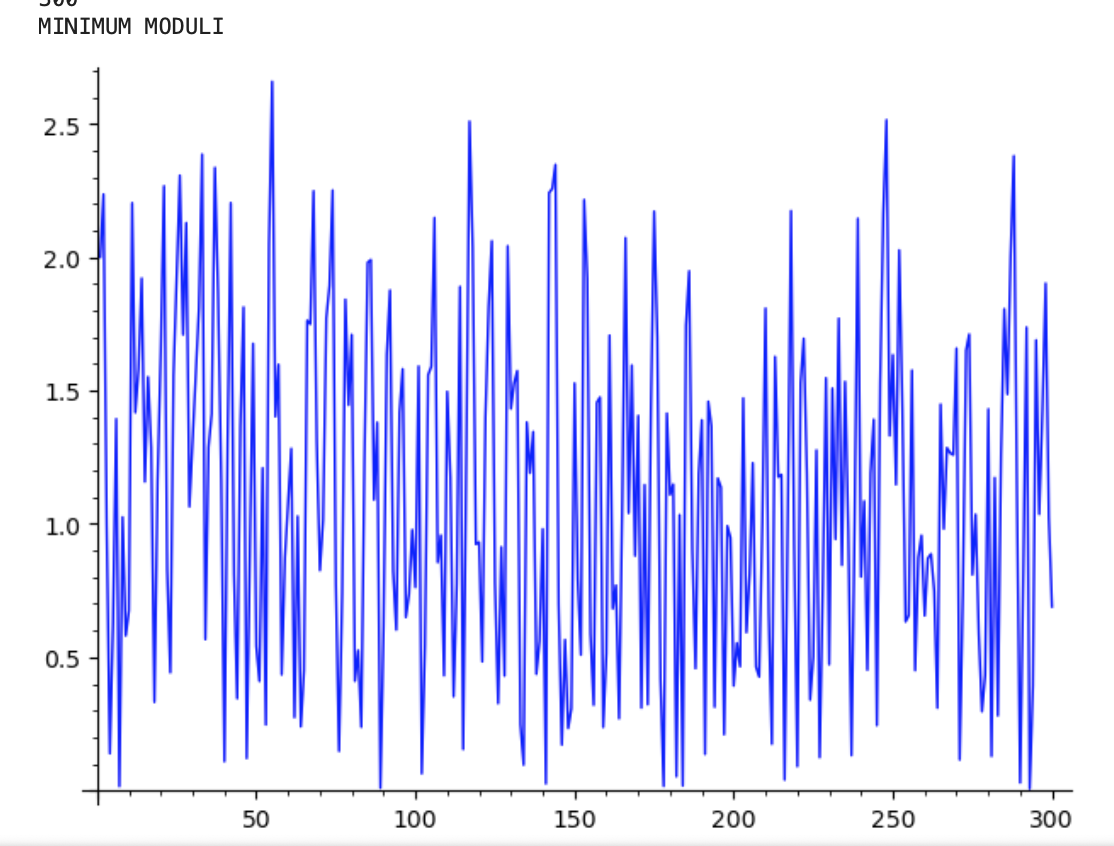}
        \caption{$h(n) = a_{p_n}$.}
        \label{fig:14a1_primes}
    \end{subfigure}
    \begin{subfigure}[t]{0.48\textwidth}
        \centering
        \includegraphics[width=\textwidth]{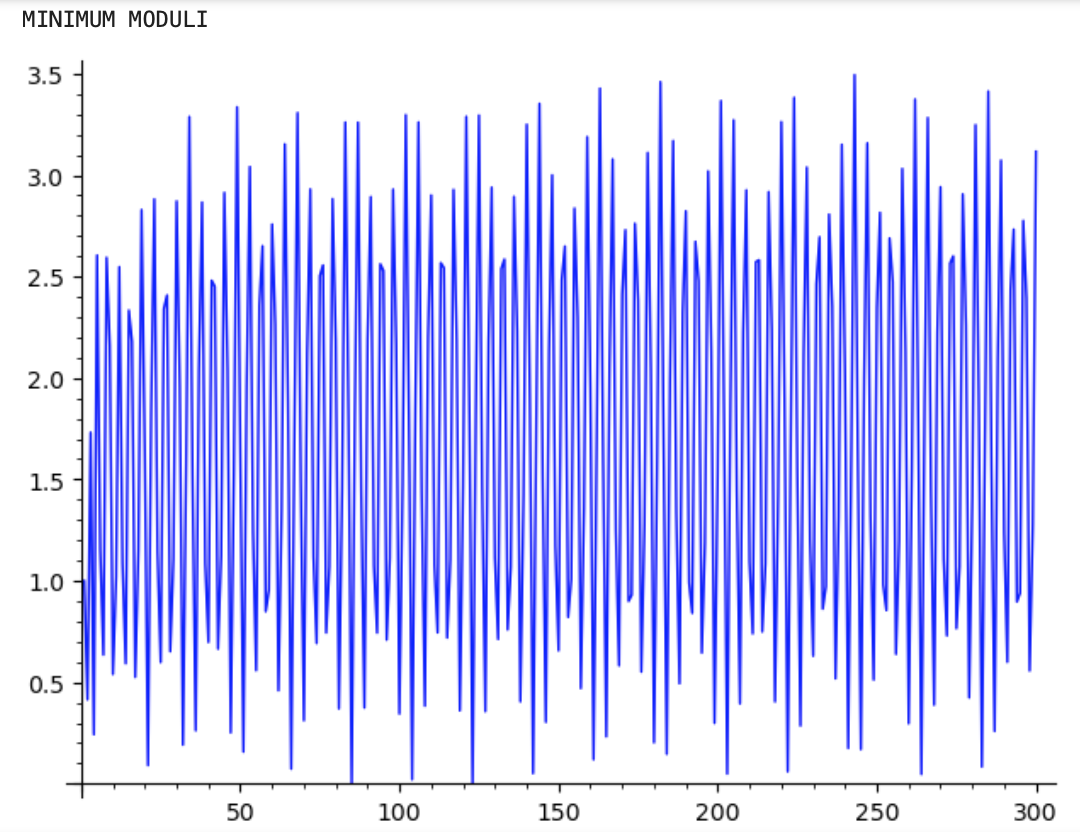}
        \caption{$h(n) = a_{1+p_n}$.}
        \label{fig:curve14a1_shift_minus_1}
    \end{subfigure}
    \caption{Curve 14a1 with $c = 1$.}
    \label{fig:crv_14a1_comparison}
\end{figure}
\begin{figure}[H]
    \centering
    \begin{subfigure}[t]{0.48\textwidth}
        \centering
       \includegraphics[width=\textwidth]{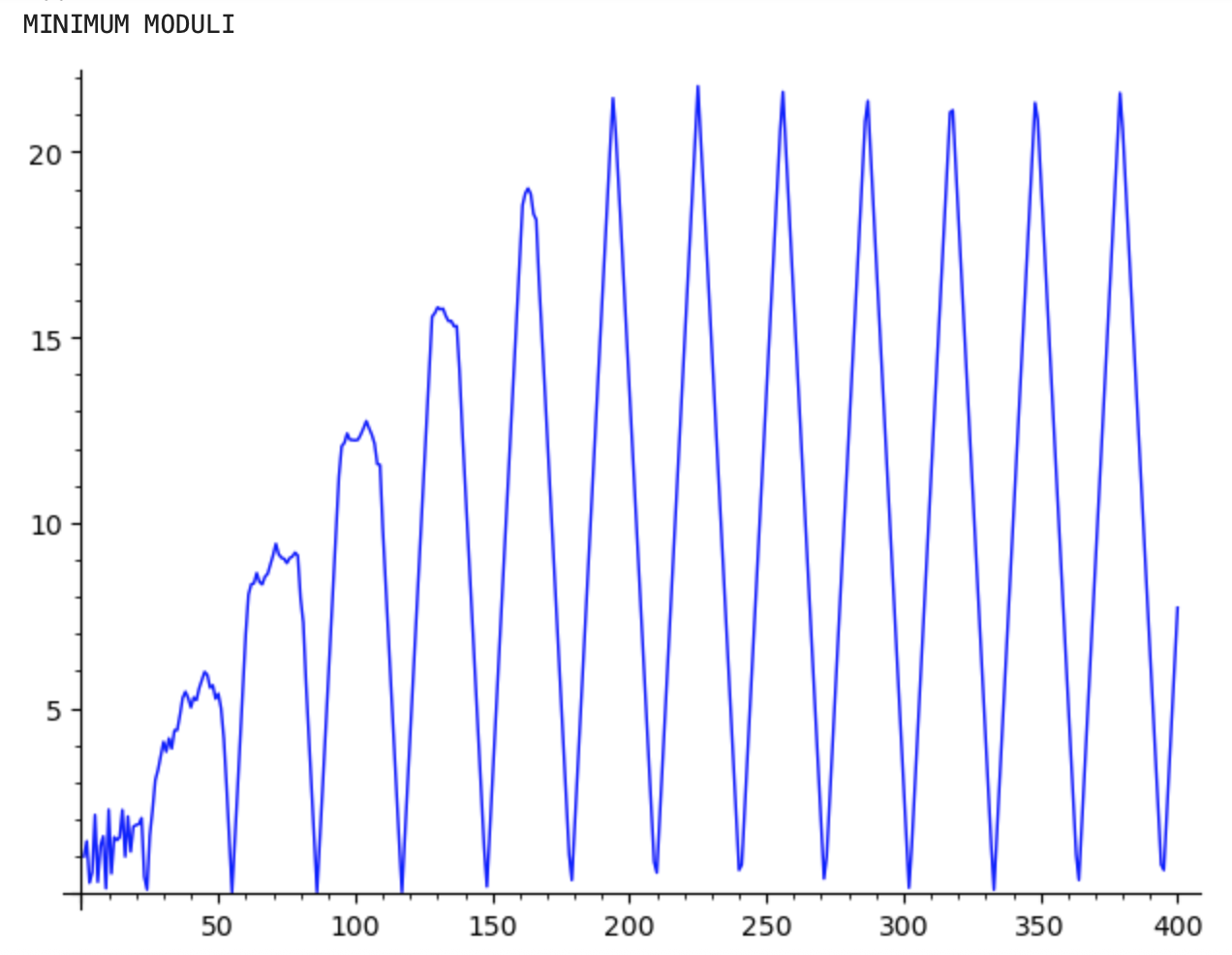}
        \caption{$h(n)=a(n)$.}
        \label{fig:curve_15a_1_all}
    \end{subfigure}
    \hfill
    \begin{subfigure}[t]{0.48\textwidth}
        \centering
        \includegraphics[width=\textwidth]{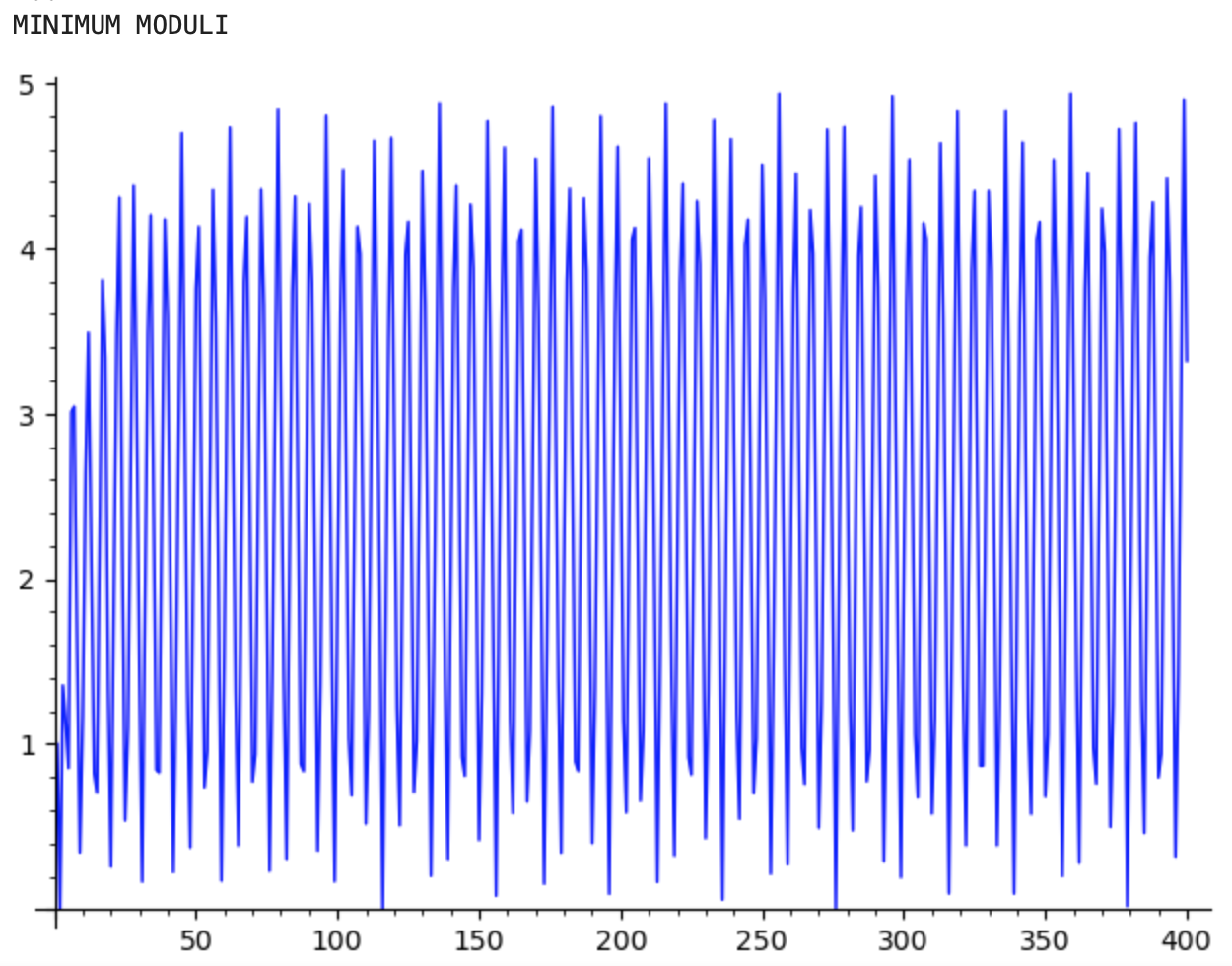}
        \caption{$h(n) = a_{1+p_n}$.}
        \label{fig:curve_15a_1_primes}
    \end{subfigure}
    \caption{Curve 15a1 with $c = 1$.}
    \label{fig:crv_15a1_comparison}
\end{figure}
\begin{figure}[H]
    \centering
    \begin{subfigure}[t]{0.48\textwidth}
        \centering
        \includegraphics[width=\textwidth]{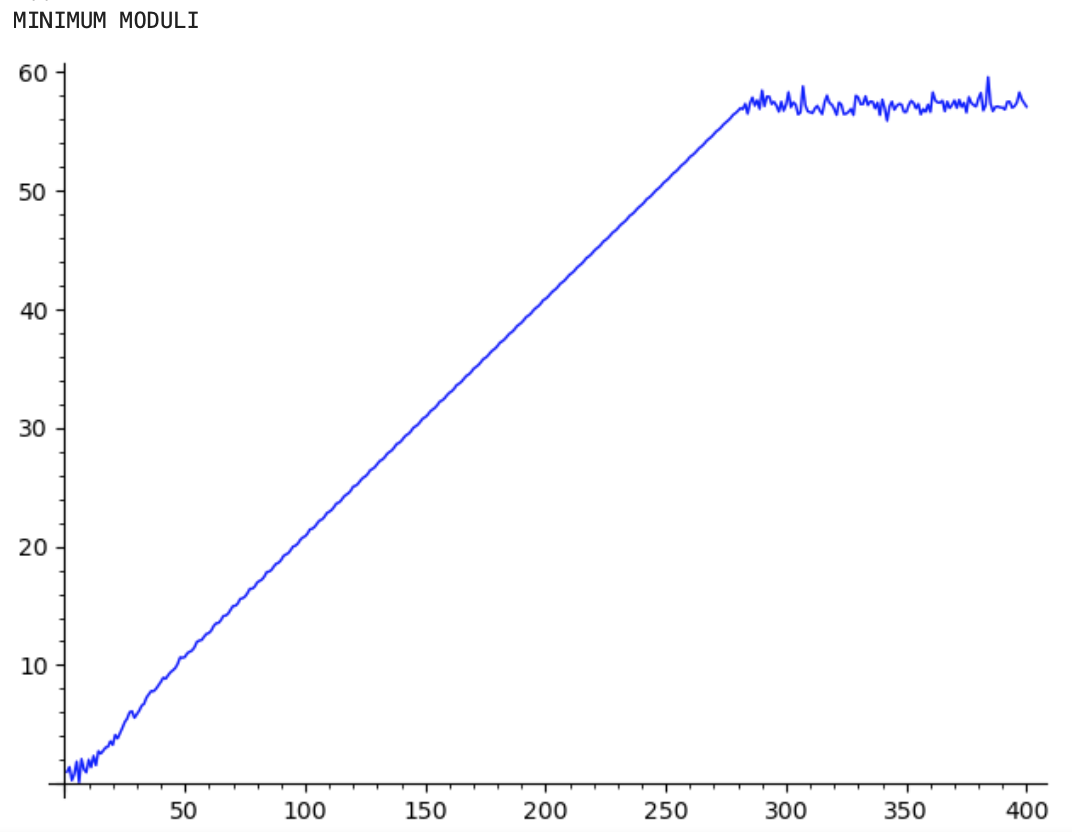}
        \caption{$h(n)=a(n)$.}
        \label{fig:crv17a1_all_n}
    \end{subfigure}
    \hfill
    \begin{subfigure}[t]{0.48\textwidth}
        \centering
        \includegraphics[width=\textwidth]{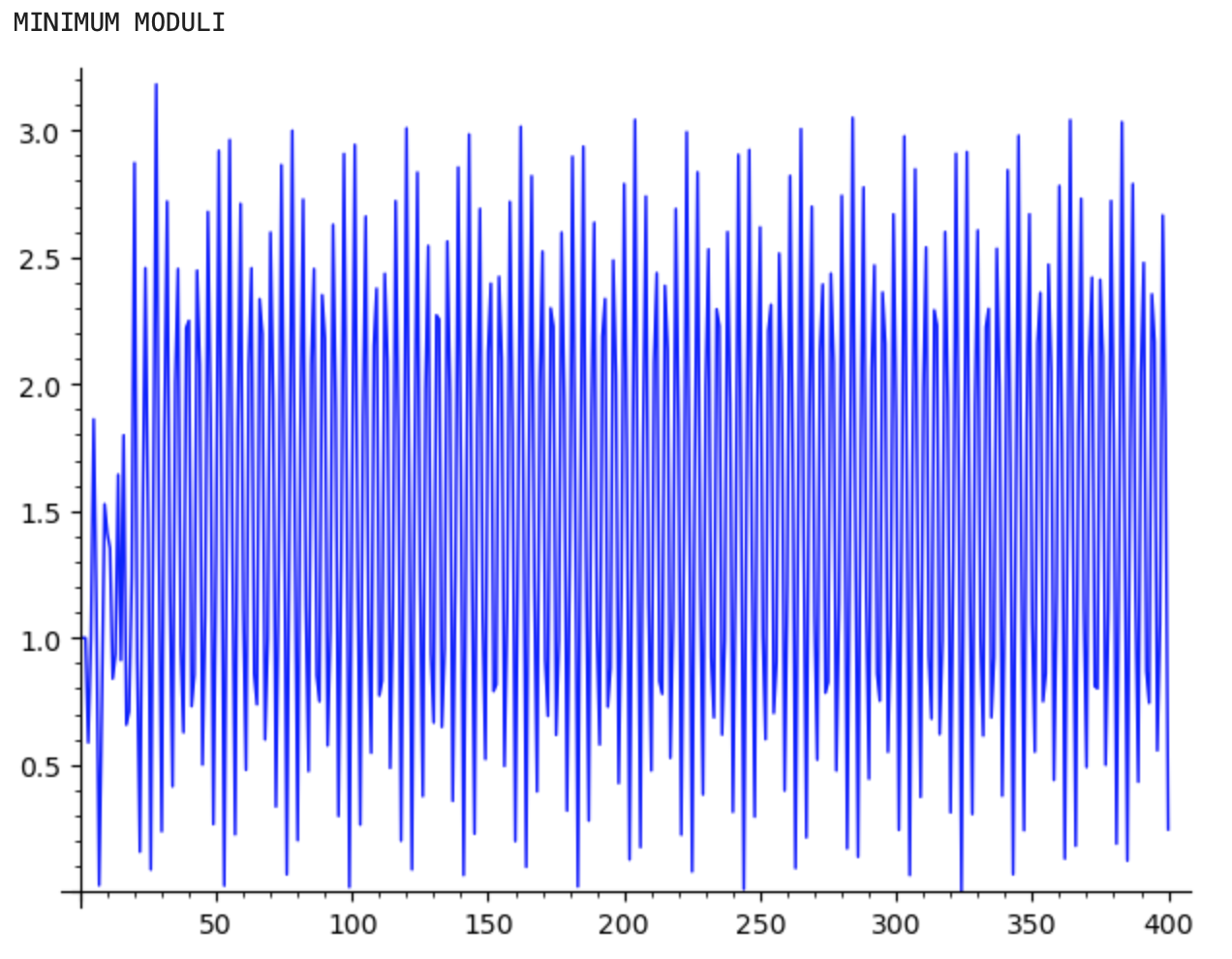}
        \caption{$h(n) = a_{1+p_n}$.}
        \label{fig:curve_17a1_on_primes}
    \end{subfigure}
    \caption{Curve 17a1 with $c = 1$.}
    \label{fig:crv_17a1_comparison}
\end{figure}
\begin{figure}[H]
    \centering
    \begin{subfigure}[t]{0.48\textwidth}
        \centering
        \includegraphics[width=\textwidth]{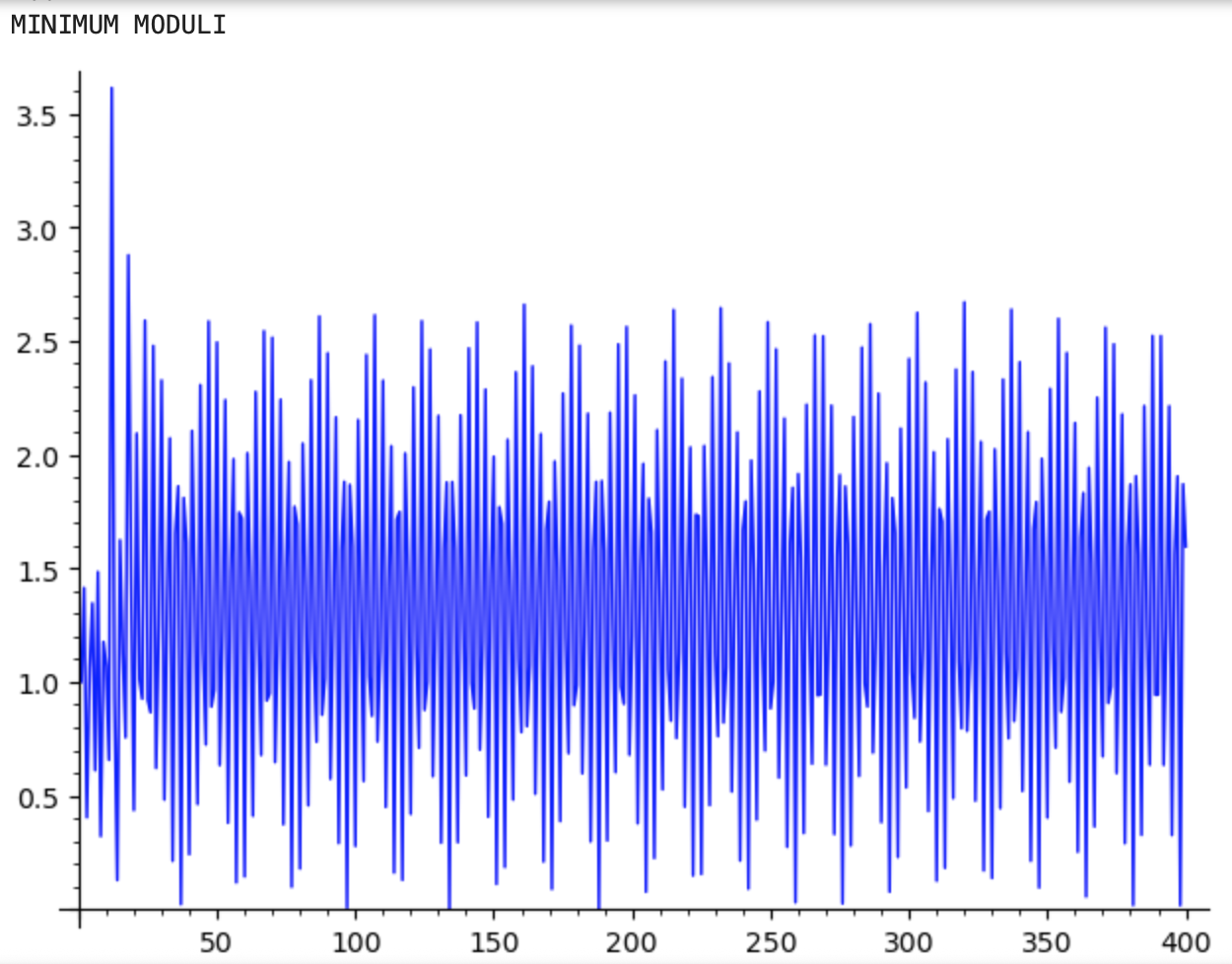}
        \caption{$h(n)=a(n)$.}
        \label{fig:curve_19a1}
    \end{subfigure}
    \hfill
    \begin{subfigure}[t]{0.48\textwidth}
        \centering
        \includegraphics[width=\textwidth]{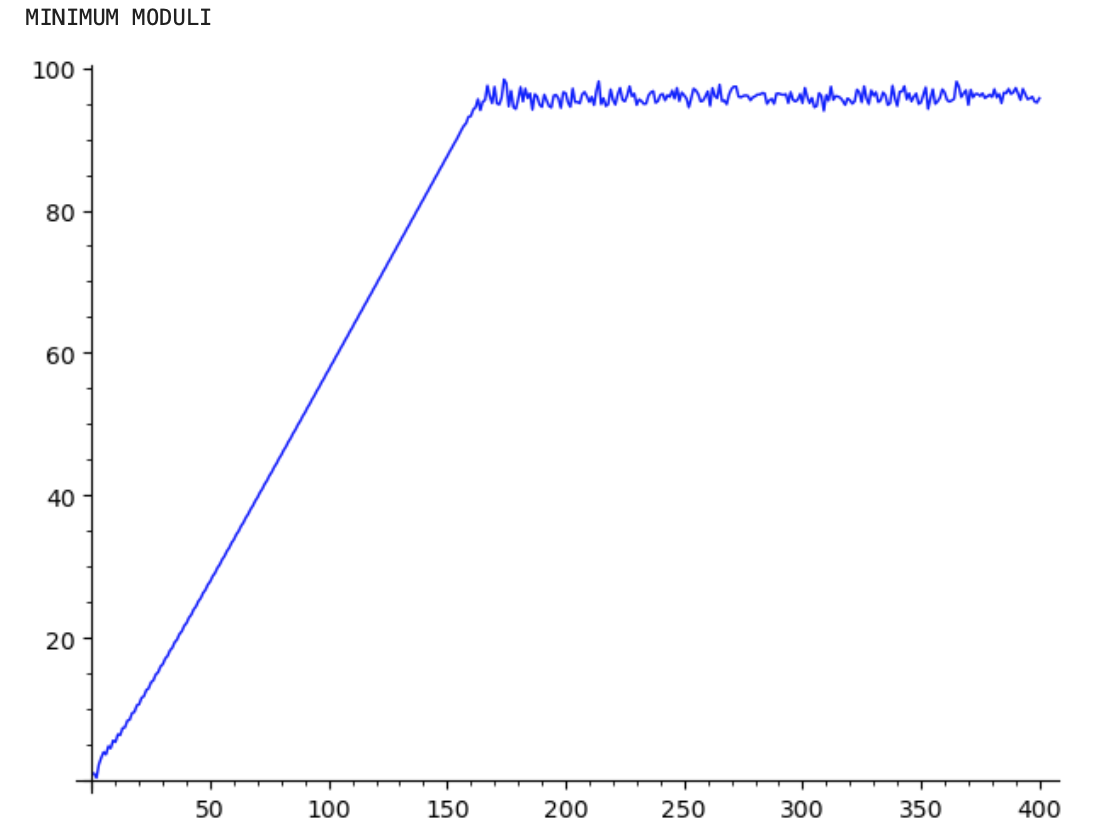}
        \caption{$h(n) = a_{1+p_n}$.}
        \label{fig:crv19a1_primes}
    \end{subfigure}
    \caption{Curve 19a1 with $c = 1$.}
    \label{fig:crv_19a1_comparison}
\end{figure}
\begin{figure}[H]
    \centering
    \begin{subfigure}[t]{0.48\textwidth}
        \centering
        \includegraphics[width=\textwidth]{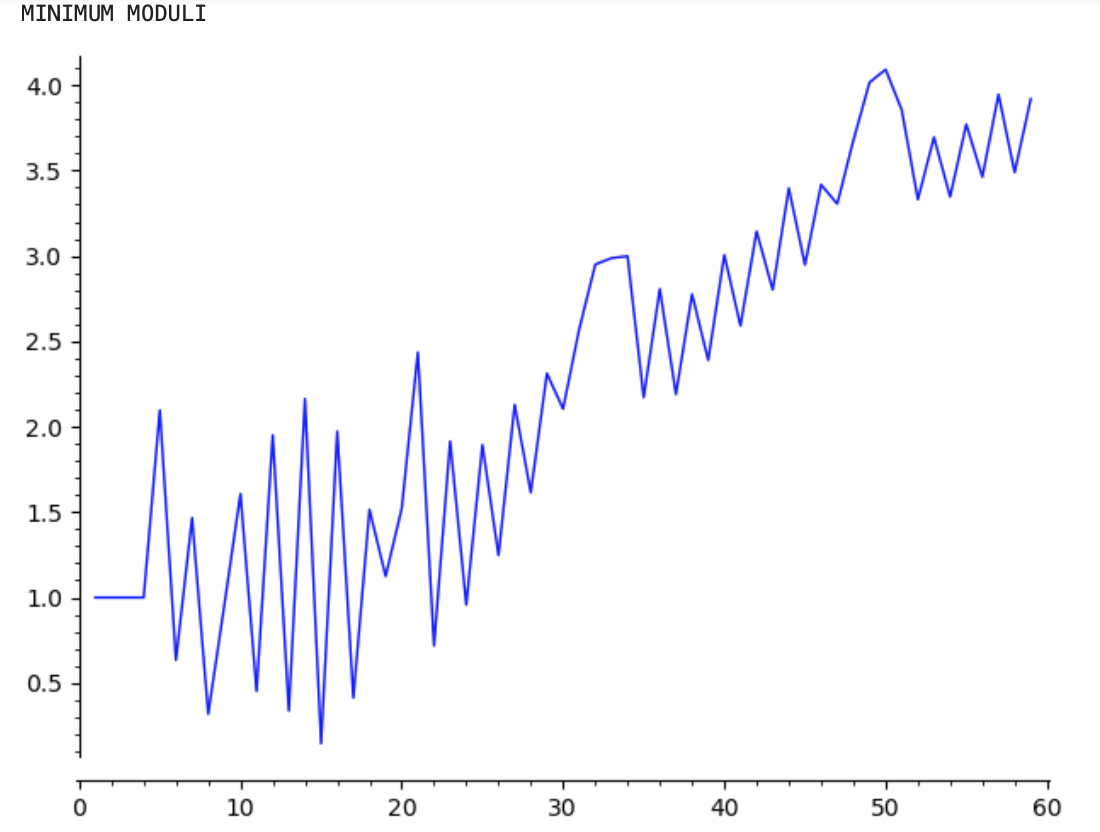}
        \caption{$h(n)=a(n)$.}
        \label{fig:crv_20a1_all}
    \end{subfigure}
    \hfill
    \begin{subfigure}[t]{0.48\textwidth}
        \centering
        \includegraphics[width=\textwidth]{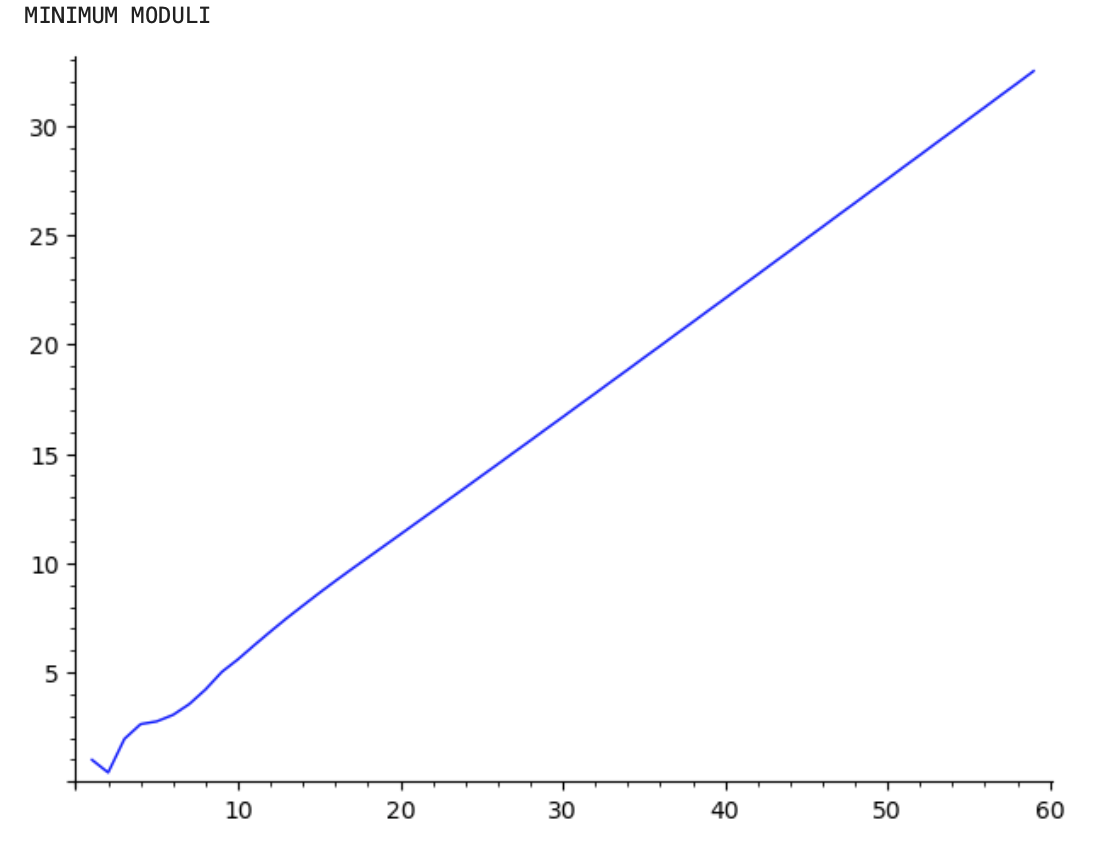}
        \caption{$h(n) = a_{1+p_n}$.}
        \label{fig:crv_20a1_primes}
    \end{subfigure}
    \caption{Curve 20a1 with $c = 1$.}
    \label{fig:crv_20a1_comparison}
\end{figure}
\begin{figure}[H]
    \centering
    \begin{subfigure}[t]{0.48\textwidth}
        \centering
        \includegraphics[width=\textwidth]{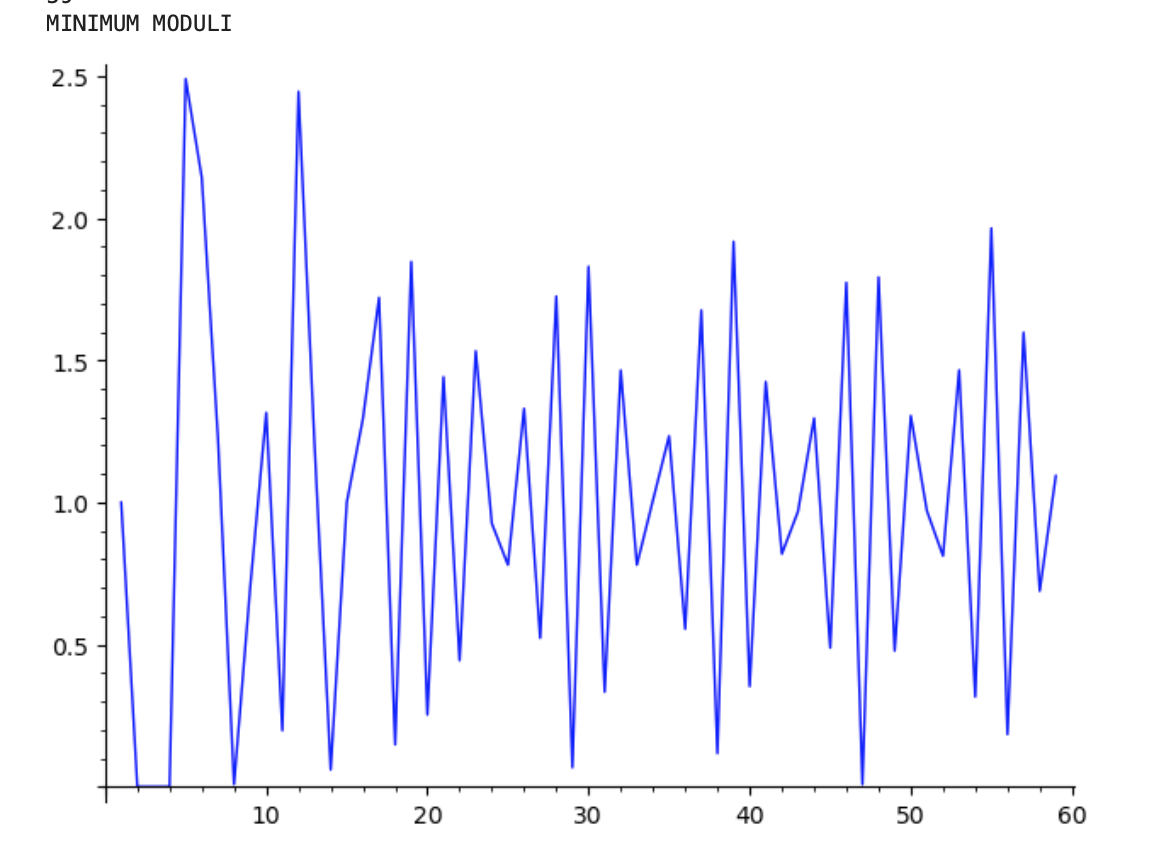}
        \caption{$h(n)=a(n)$.}
        \label{fig:crv21a1_all}
    \end{subfigure}
    \hfill
    \begin{subfigure}[t]{0.48\textwidth}
        \centering
        \includegraphics[width=\textwidth]{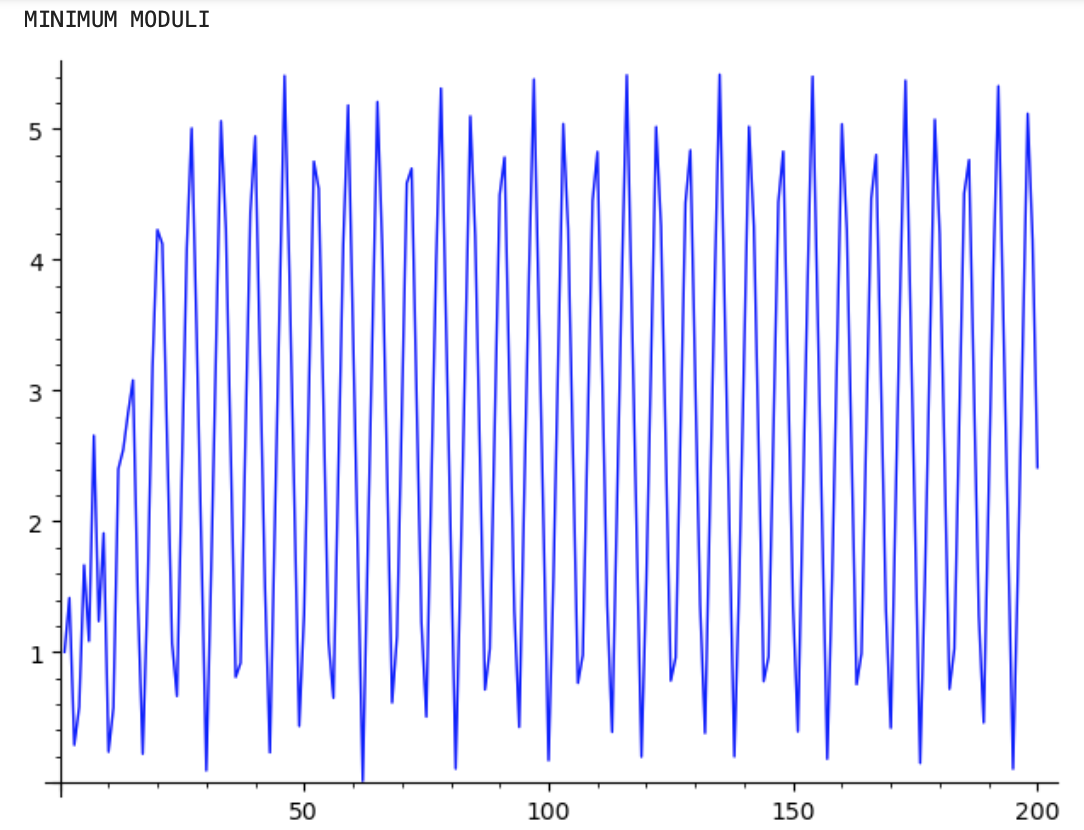}
        \caption{$h(n) = a_{1+p_n}$.}
        \label{fig:crv21a1_primes}
    \end{subfigure}
    \caption{Curve 21a1 with $c = 1$.}
    \label{fig:crv_21a1_primes}
\end{figure}
\begin{figure}[H]
    \centering
    \begin{subfigure}[t]{0.48\textwidth}
        \centering
        \includegraphics[width=\textwidth]{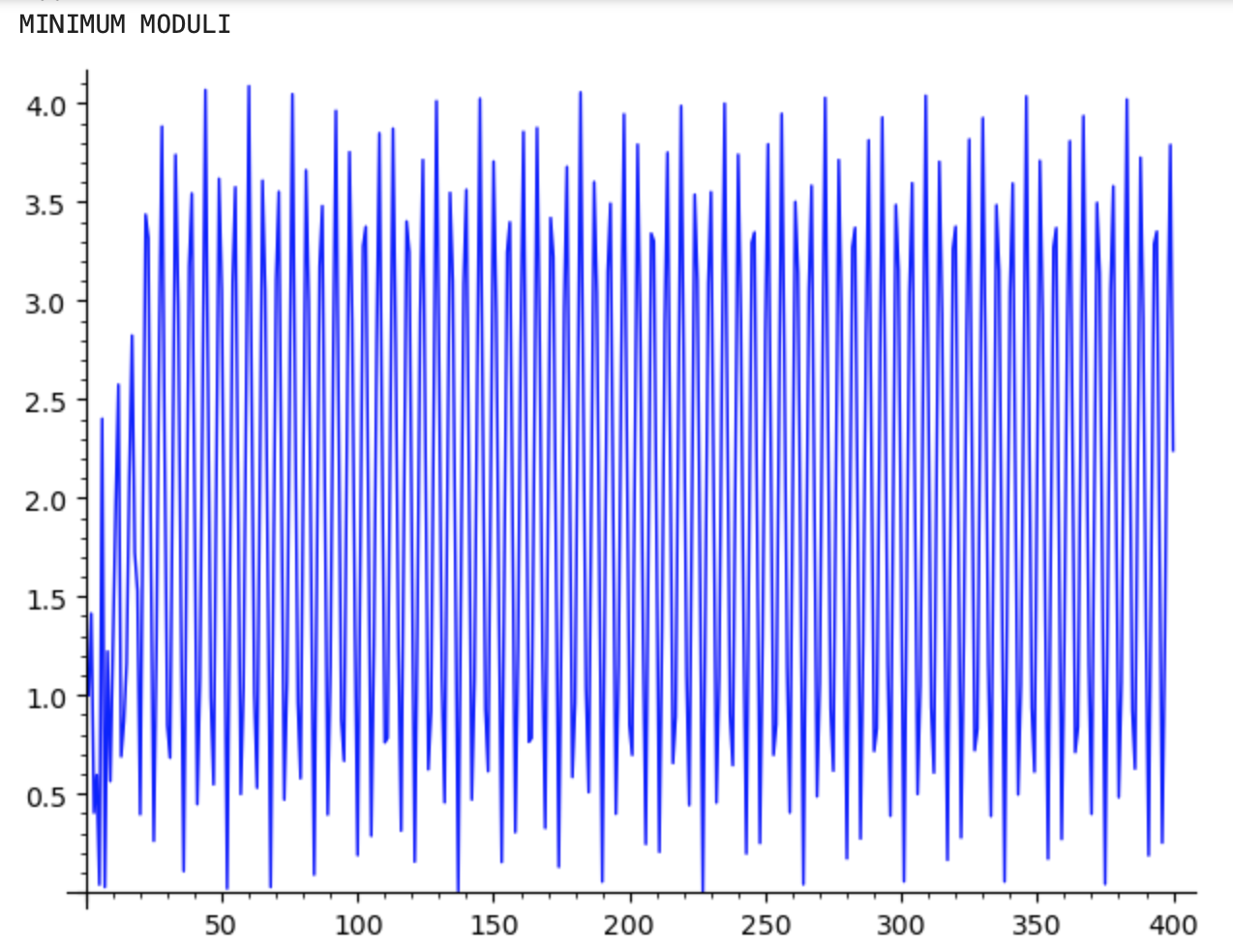}
        \caption{$h(n)=a(n)$.}
        \label{fig:crv24a1_all_n}
    \end{subfigure}
    \hfill
    \begin{subfigure}[t]{0.48\textwidth}
        \centering
        \includegraphics[width=\textwidth]{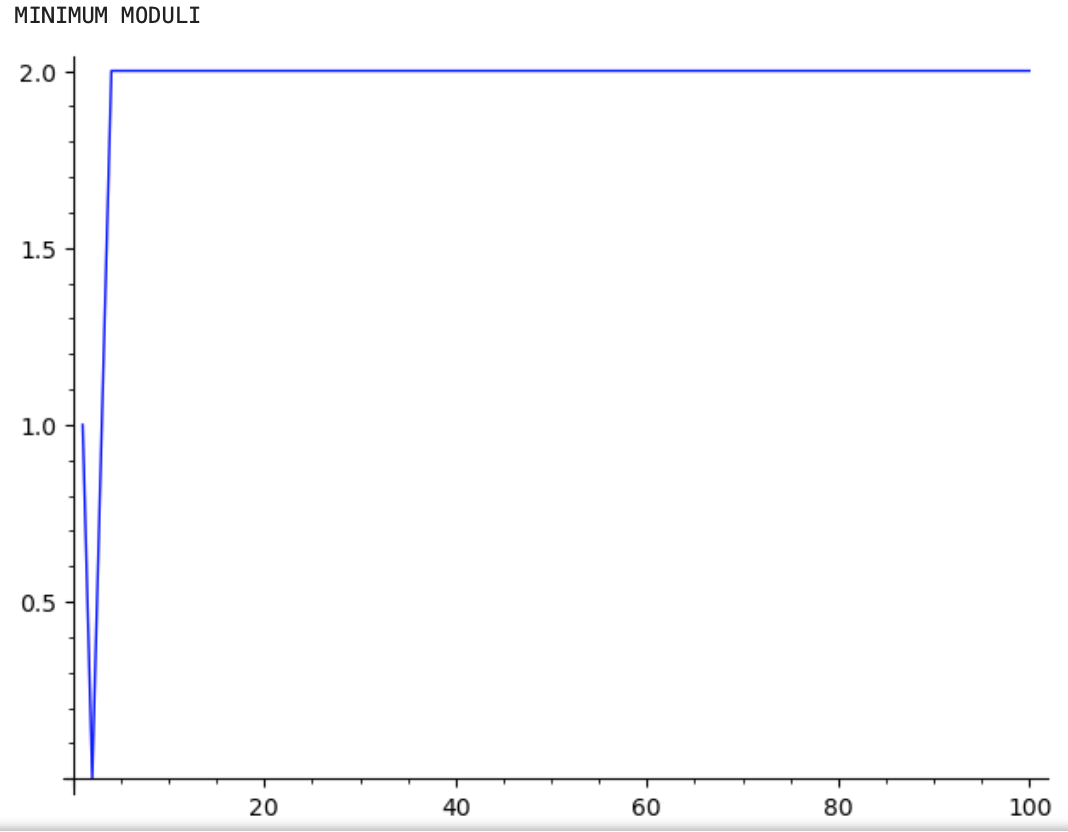}
        \caption{$h(n) = a_{1+p_n}$.}
        \label{fig:crv24a1_primes}
    \end{subfigure}
    \caption{Curve 24a1 with $c = 1$.}
    \label{fig:crv_24a1_comparison}
\end{figure}

\begin{figure}[H]
    \centering
    \begin{subfigure}[t]{0.48\textwidth}
        \centering
        \includegraphics[width=\textwidth]{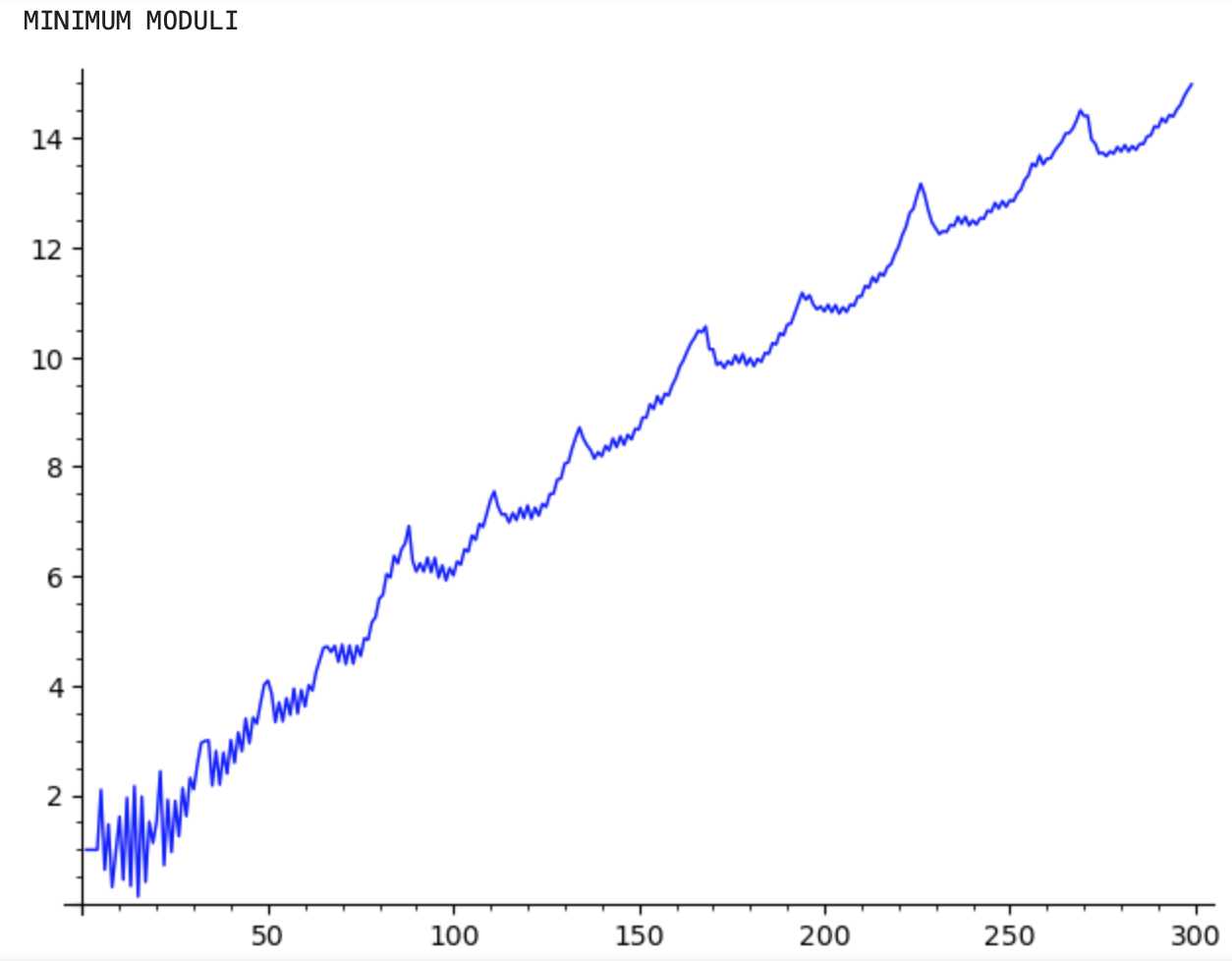}
        \caption{$h(n)=a(n)$.}
        \label{fig:crv26a1_all}
    \end{subfigure}
    \hfill
    \begin{subfigure}[t]{0.48\textwidth}
        \centering
        \includegraphics[width=\textwidth]{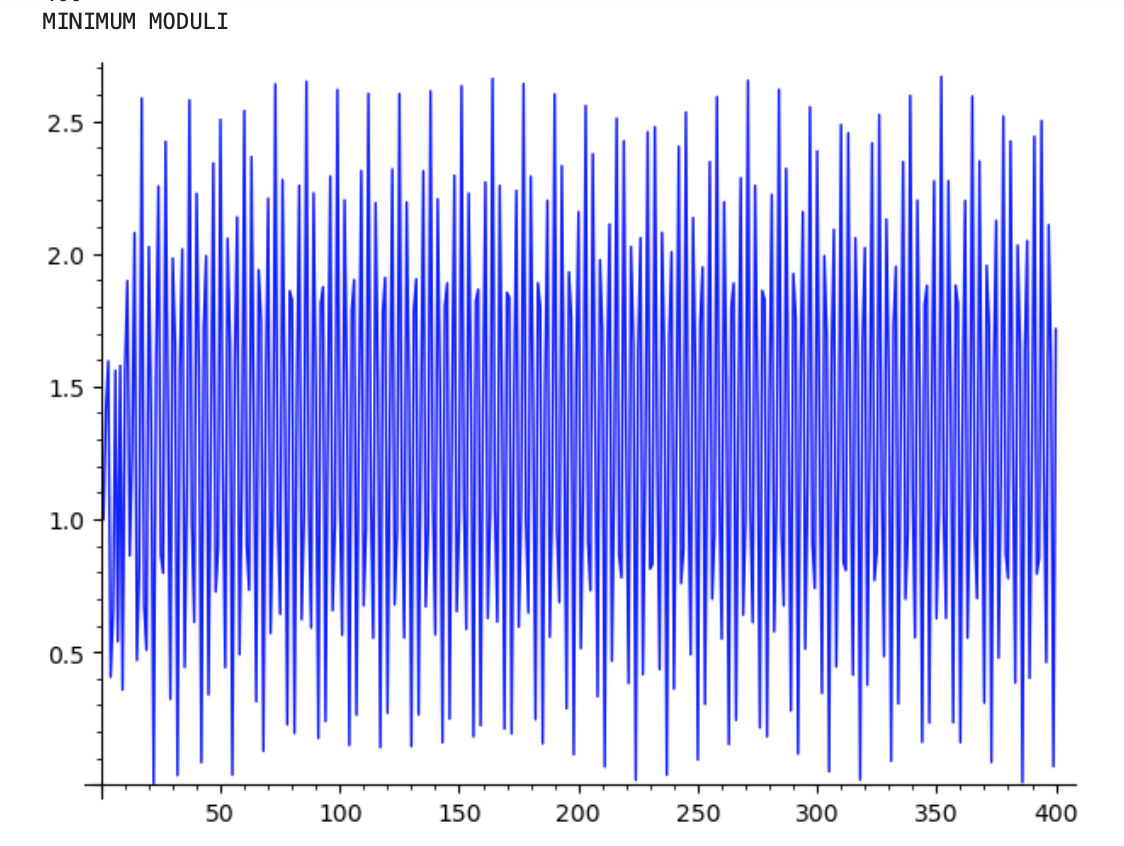}
        \caption{$h(n) = a_{1+p_n}$.}
        \label{fig:crv26a1_primes}
    \end{subfigure}
    \caption{Curve 26a1 with $c = 1$.}
    \label{fig:crv_26a1_comparison}
\end{figure}

\begin{figure}[H]
    \centering
    \begin{subfigure}[t]{0.48\textwidth}
        \centering
        \includegraphics[width=\textwidth]{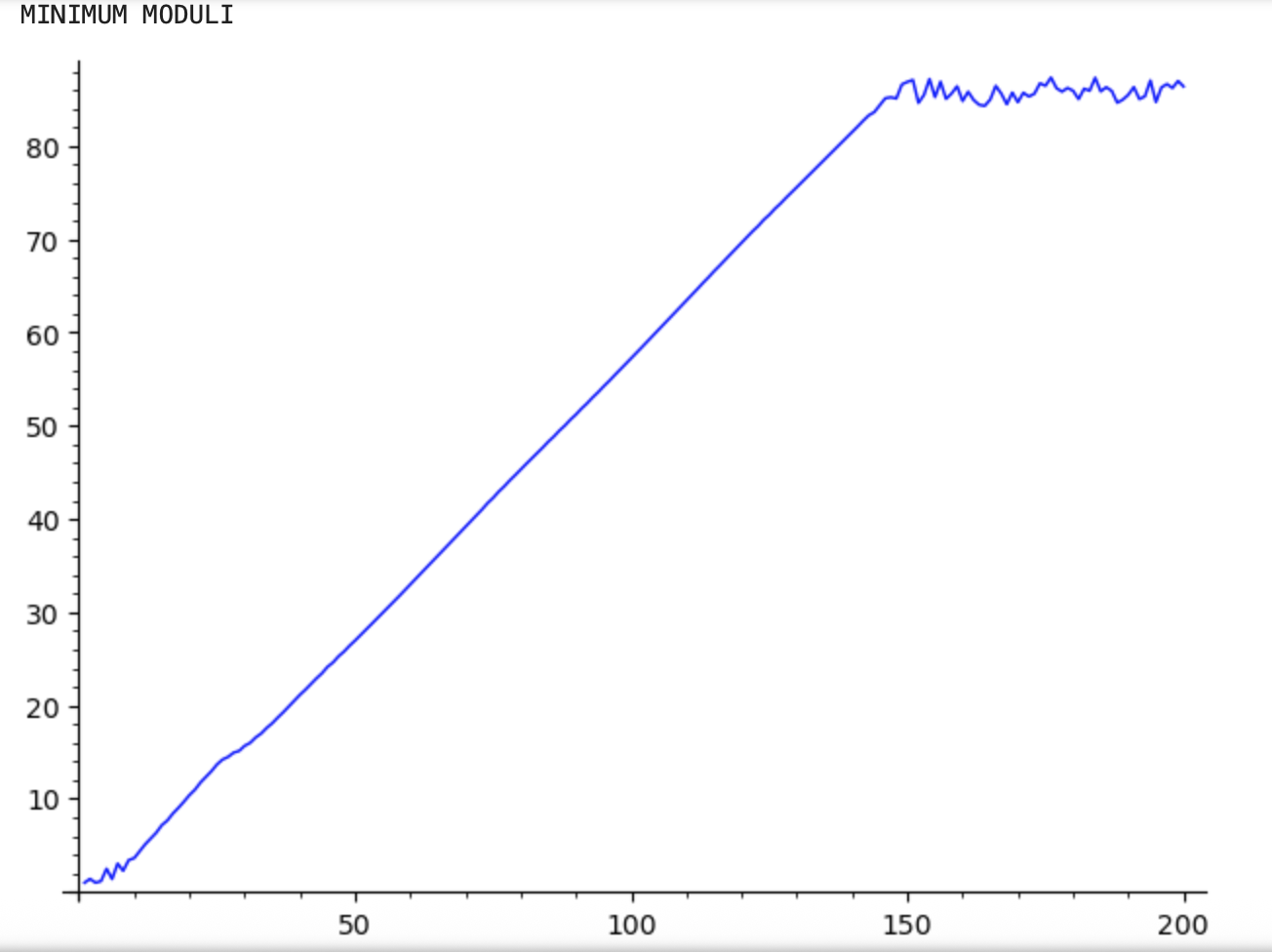}
        \caption{$h(n)=a(n)$.}
        \label{fig:crv26b1_all}
    \end{subfigure}
    \hfill
    \begin{subfigure}[t]{0.48\textwidth}
        \centering
        \includegraphics[width=\textwidth]{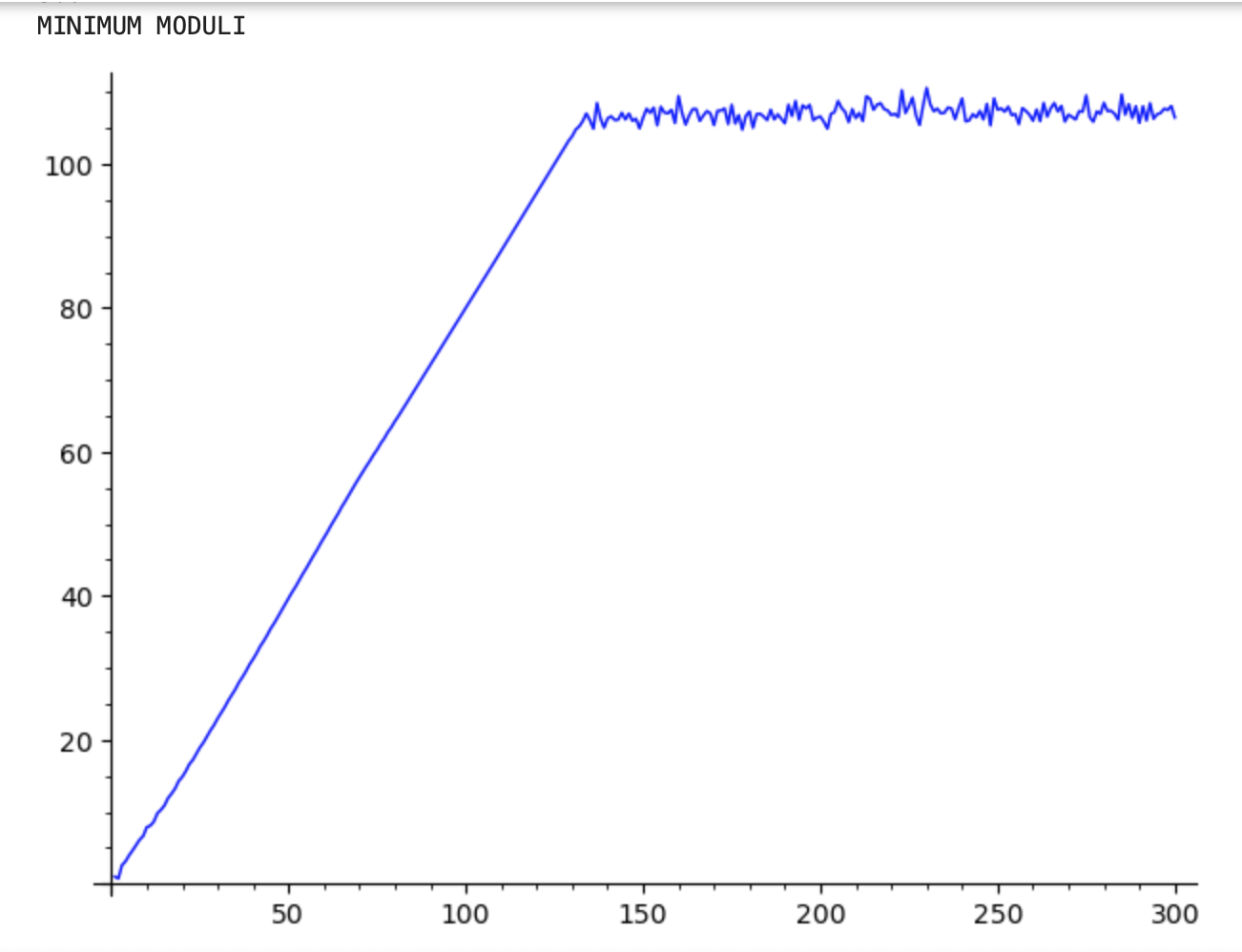}
        \caption{$h(n) = a_{1+p_n}$.}
        \label{fig:crv26b1_primes}
    \end{subfigure}
    \caption{Curve 26b1 with $c = 1$.}
    \label{fig:crv_26b1_comparison}
\end{figure}

\begin{figure}[H]
    \centering
    \begin{subfigure}[t]{0.48\textwidth}
        \centering
        \includegraphics[width=\textwidth]{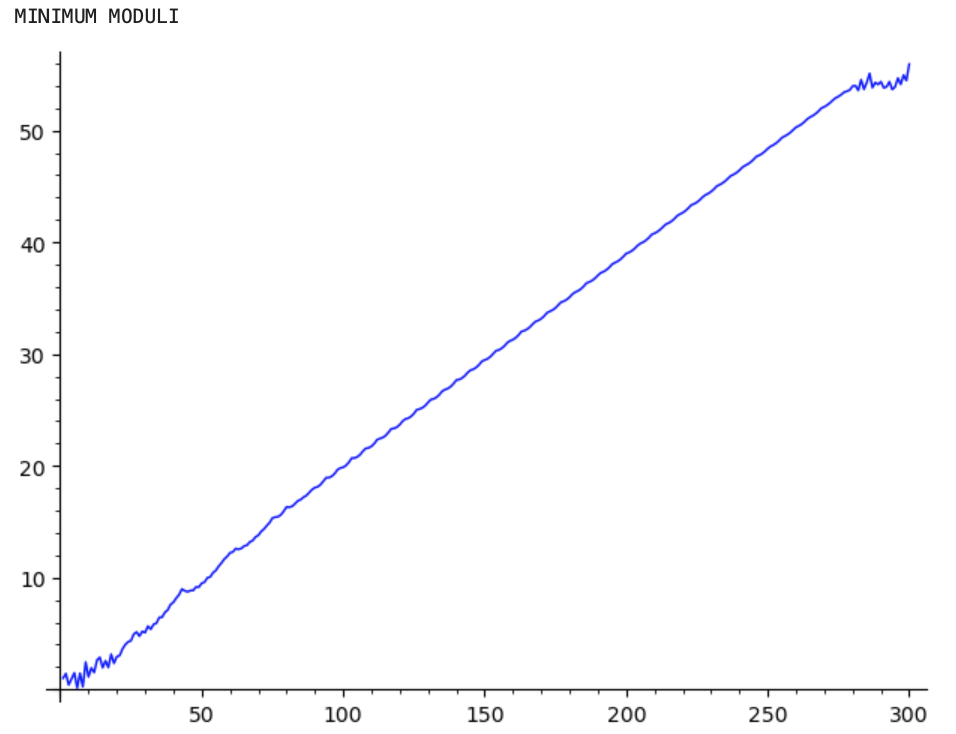}
        \caption{$h(n)=a(n)$.}
        \label{fig:crv27a1_all}
    \end{subfigure}
    \hfill
    \begin{subfigure}[t]{0.48\textwidth}
        \centering
        \includegraphics[width=\textwidth]{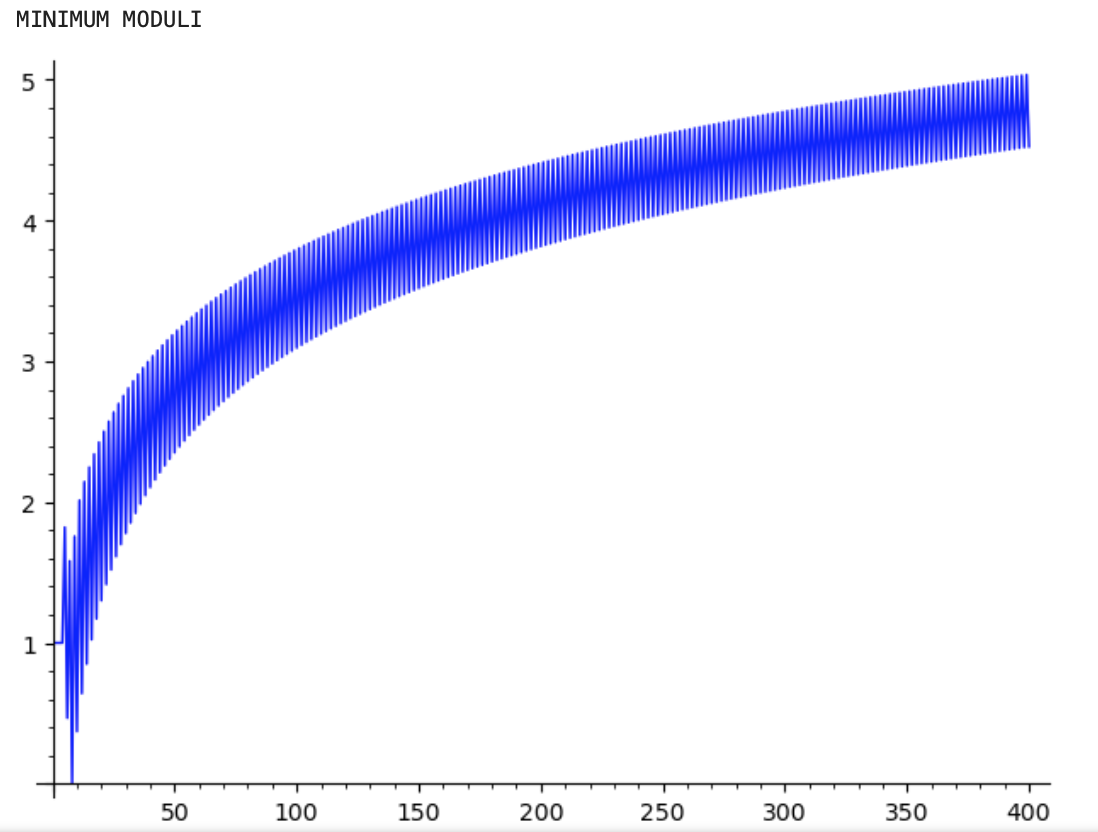}
        \caption{$h(n) = a_{1+p_n}$.}
        \label{fig:crv27a1_primes}
    \end{subfigure}
    \caption{Curve 27a1 with $c = 1$.}
    \label{fig:crv27a1_comparison}
\end{figure}

\begin{figure}[H]
    \centering
    \begin{subfigure}[t]{0.48\textwidth}
        \centering
        \includegraphics[width=\textwidth]{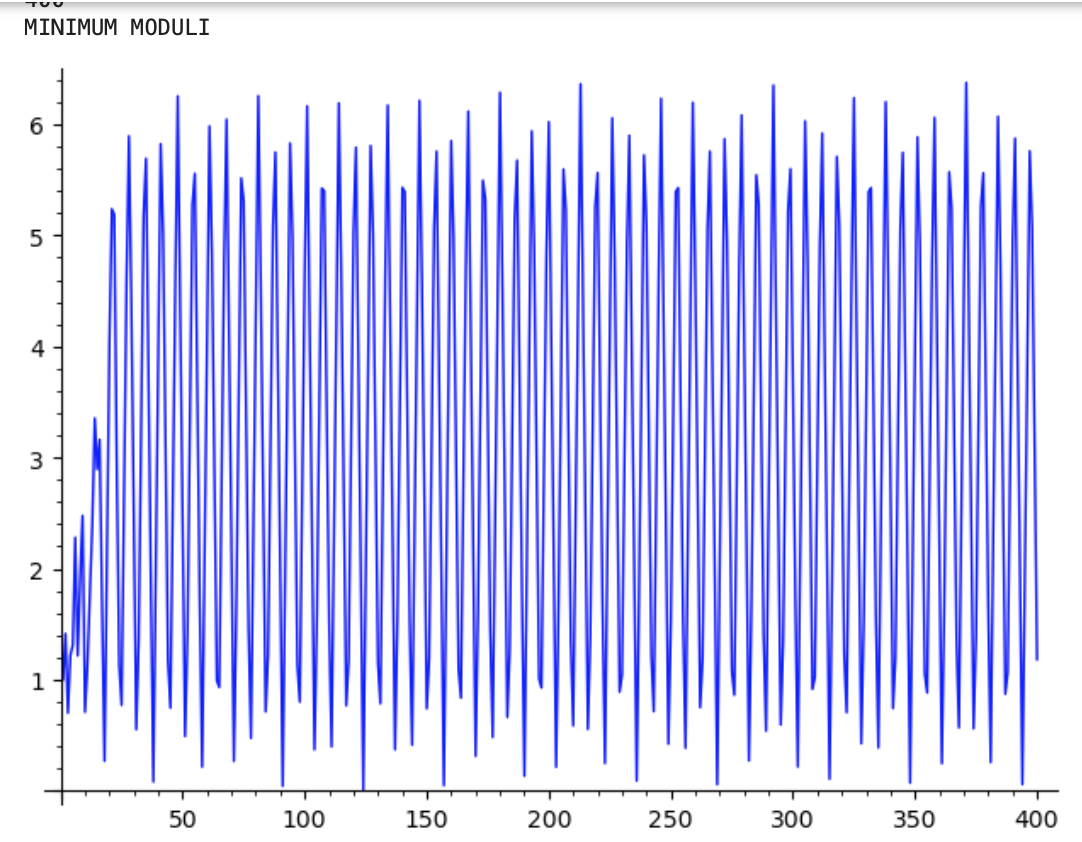}
        \caption{$h(n)=a(n)$.}
        \label{fig:curve_37a1}
    \end{subfigure}
    \hfill
    \begin{subfigure}[t]{0.48\textwidth}
        \centering
        \includegraphics[width=\textwidth]{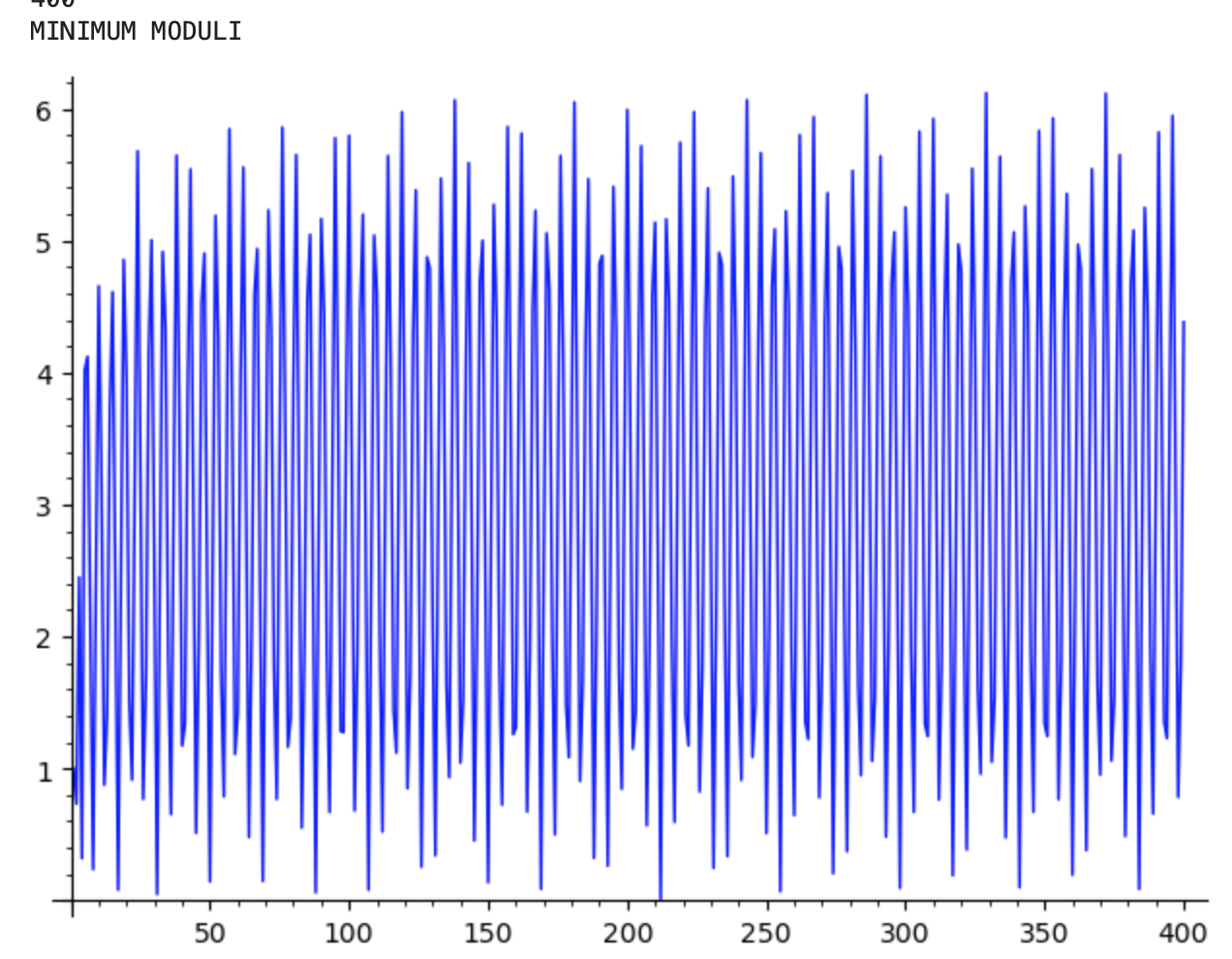}
        \caption{$h(n) = a_{1+p_n}$.}
        \label{fig:crv_37a1_primes}
    \end{subfigure}
    \caption{Curve 37a1 with $c = 1$.}
    \label{fig:crv_37a1_comparison}
\end{figure}
\begin{figure}[H]
    \centering
    \begin{subfigure}[t]{0.48\textwidth}
        \centering
        \includegraphics[width=\textwidth]{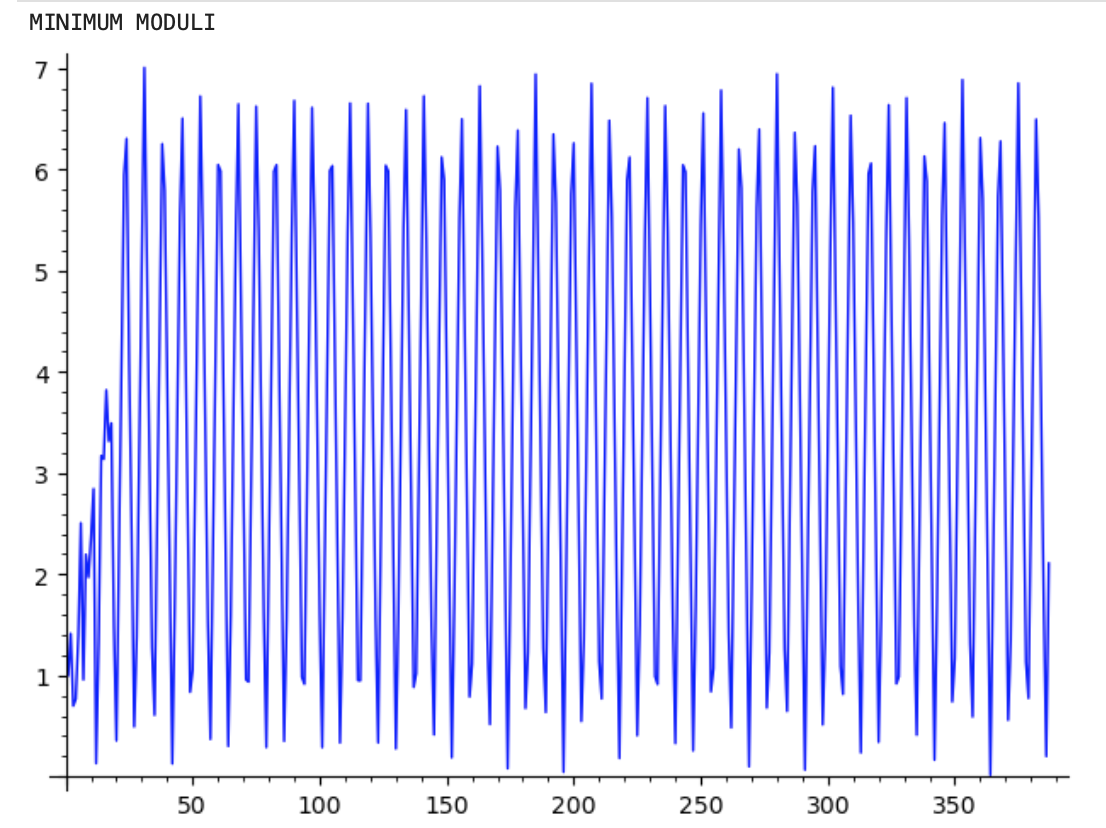}
        \caption{$h(n)=a(n)$.}
        \label{fig:curve_43a1}
    \end{subfigure}
    \hfill
    \begin{subfigure}[t]{0.48\textwidth}
        \centering
        \includegraphics[width=\textwidth]{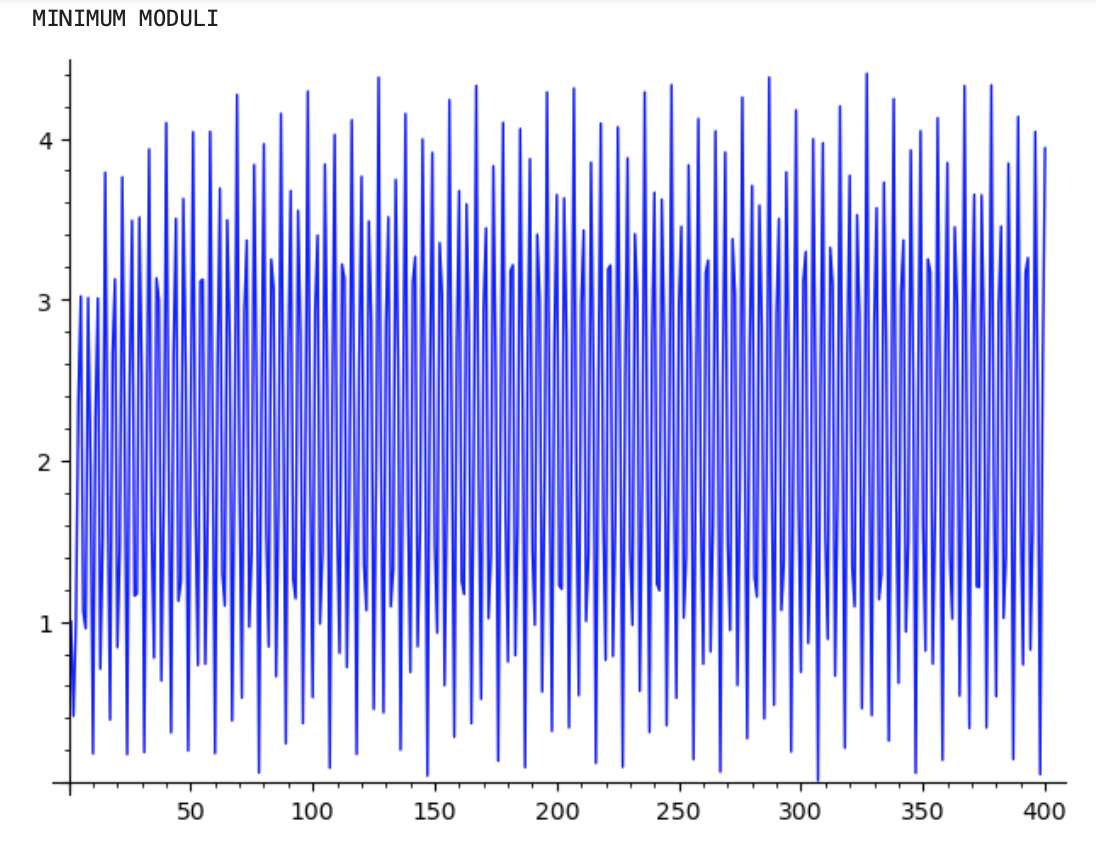}
        \caption{$h(n) = a_{1+p_n}$.}
        \label{fig:crv43a1_primes}
    \end{subfigure}
    \caption{Curve 43a1 with $c = 1$.}
    \label{fig:crv43a1_comparison}
\end{figure}
\begin{figure}[H]
    \centering
    \begin{subfigure}[t]{0.48\textwidth}
        \centering
        \includegraphics[width=\textwidth]{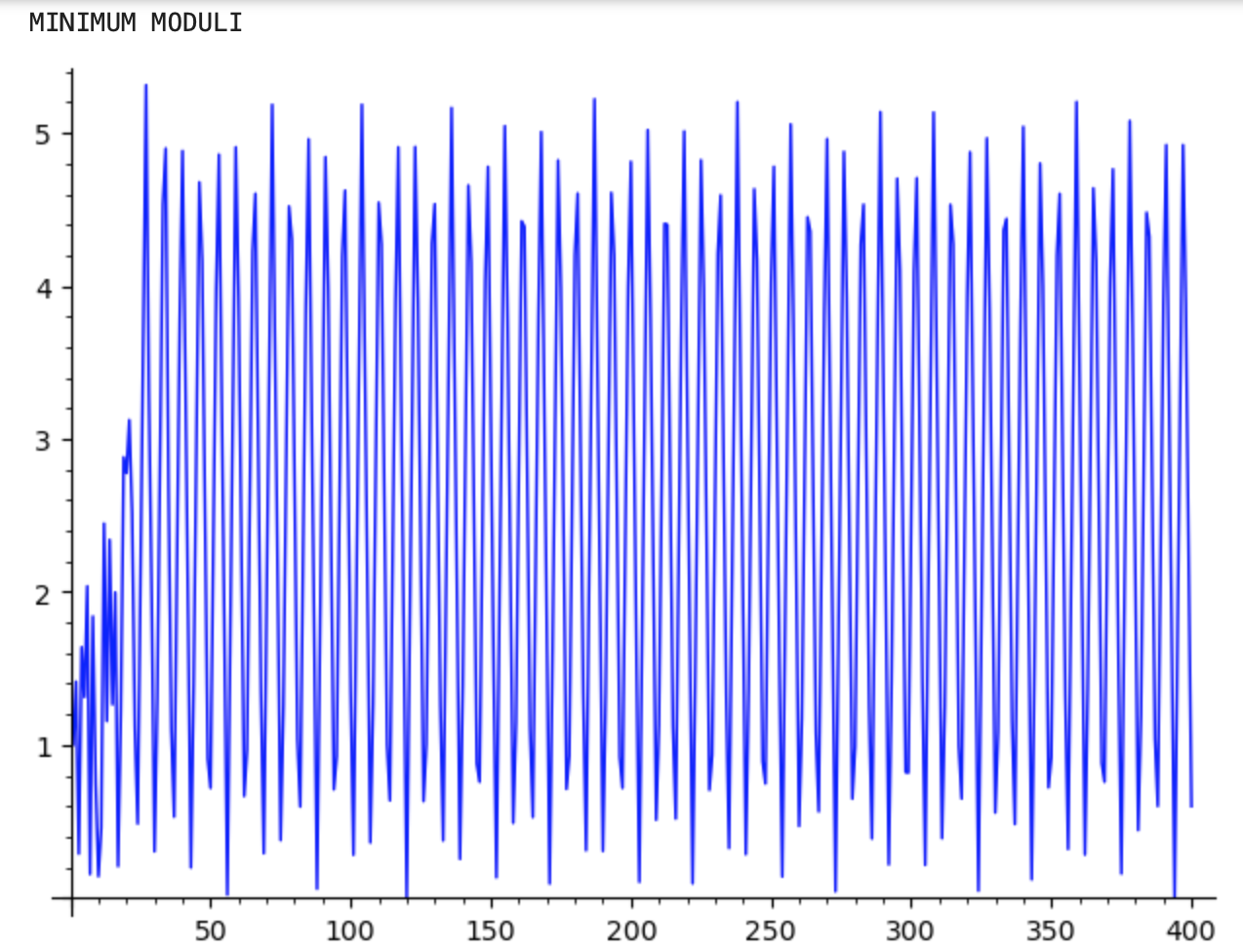}
        \caption{$h(n)=a(n)$.}
        \label{fig:curve_53a1}
    \end{subfigure}
    \hfill
    \begin{subfigure}[t]{0.48\textwidth}
        \centering
        \includegraphics[width=\textwidth]{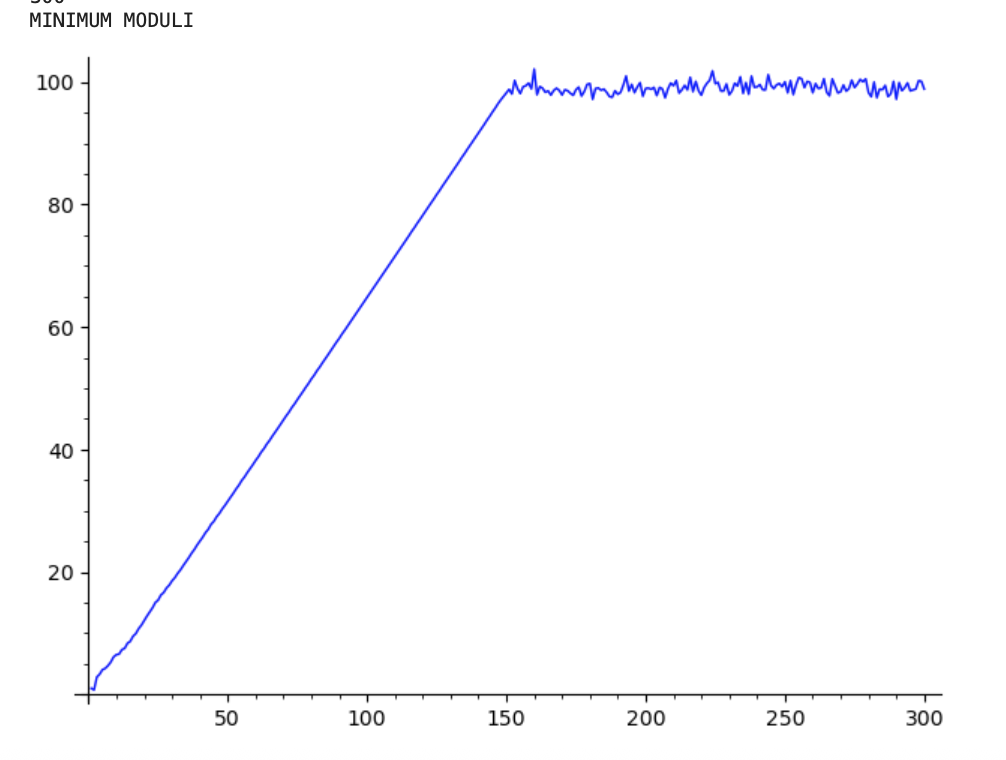}
        \caption{$h(n) = a_{1+p_n}$.}
        \label{fig:crv53a1_primes}
    \end{subfigure}
    \caption{Curve 53a1 with $c = 1$.}
    \label{fig:crv53a1_comparison}
\end{figure}
\begin{figure}[H]
    \centering
    \includegraphics[width=1\textwidth]{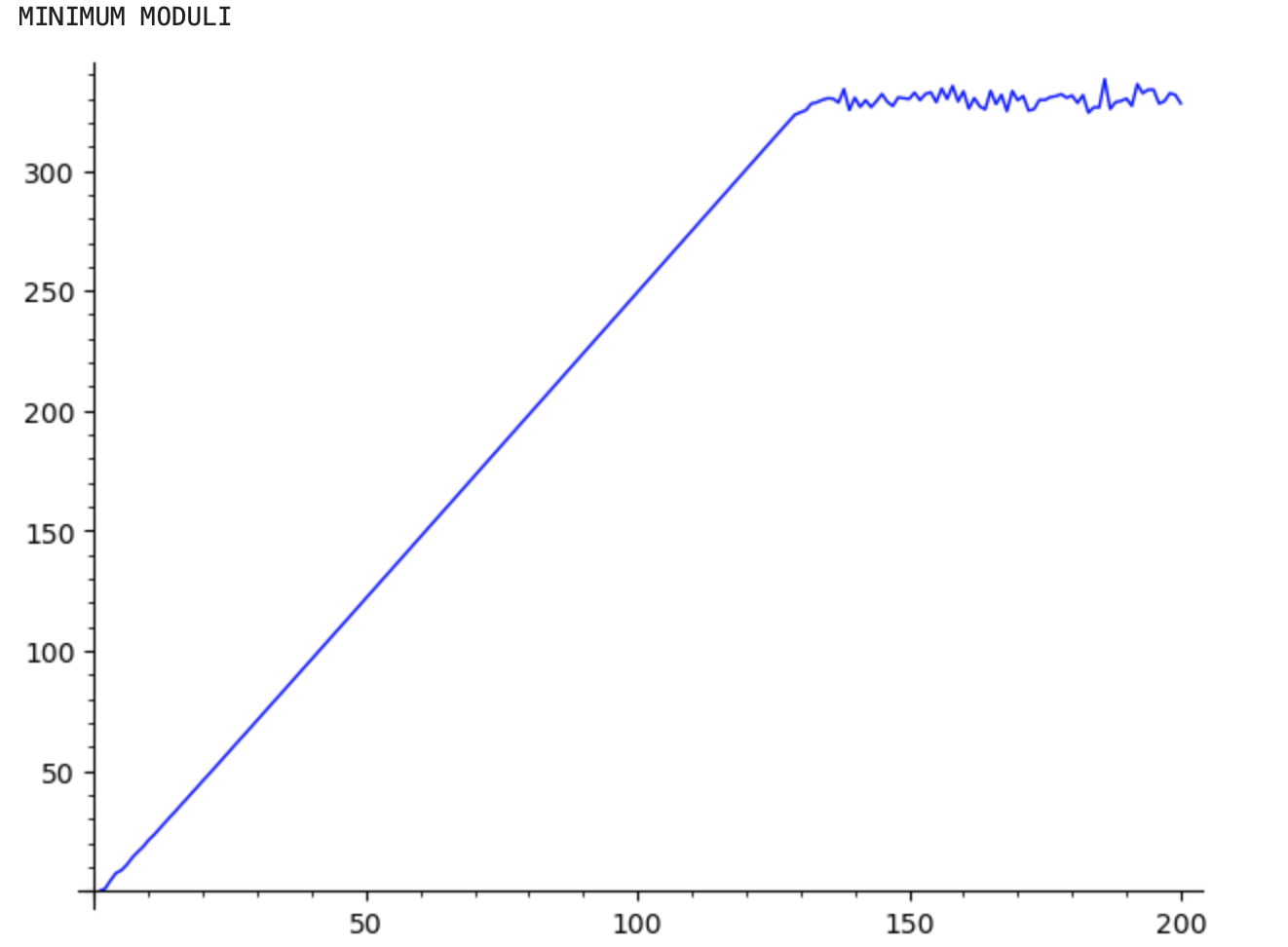}
    \caption{Minimum moduli for $h(n)=\tau(n)$ with $c=0$.}
    \label{fig:def_tau}
\end{figure}


\begin{thebibliography}{99}
\bibitem[A22]{A22}
E.~P.~Adams, \emph{Smithsonian Mathematical Formulae and Tables of Elliptic Functions, 
Smithsonian Institution, Washington, D.C., 1922.}
\bibitem[Anthropic2026]{Anthropic2026}
Anthropic, \textit{Claude Sonnet 4.5, 4.6}, 2026.
\url{https://claude.ai}. 
\bibitem[Bre24]{Bre24} B. Brent, 
 \textit{Polynomial interpolation of modular forms for Hecke groups},
 {\em Integers}, {\bf{21} \rm (\#A118) (2021),
\url{https://math.colgate.edu/~integers/v118/v118.pdf}.}
\bibitem[Bre25]{Bre25} B. Brent, \textit{tau\_properties}, GitHub repository,\url{https://github.com/barry314159a/tau\_properties}, 2025.
\bibitem[Bre25a]{Bre25a}
B. Brent, \emph{Matrix methods for arithmetic functions}, 
available at \url{https://arxiv.org/pdf/math/0405083}.
\bibitem[cremona]{cremona}
J.~E. Cremona,
\newblock \emph{The Elliptic Curve Database for Conductors to 500000},
\newblock available at \url{https://johncremona.github.io/ecdata/}, 2024,
accessed February 2026.
\bibitem[diamond-shurman]{diamond-shurman}
F.~Diamond and J.~Shurman,
\newblock \emph{A First Course in Modular Forms},
\newblock Graduate Texts in Mathematics, Vol.~228, Springer, 2005.
\bibitem[F]{F}
D.~Festi,
\textit{Notes on Elliptic Curves},
lecture notes, July 22, 2022.
Available at
\url{https://nesinkoyleri.org/wp-content/uploads/2022/01/Notes_on_elliptic_curves.pdf}.
\bibitem[Gh11]{Gh11}
M.~Gholami,
``Division of two power-series by a determinant,''
\emph{Int. J. Contemp. Math. Sciences},
vol.~6, no.~15, pp.~735--744, 2011.
\bibitem[Go99]{Go99} H. W. Gould, \textit{The Girard-Waring power sum formulas for symmetric functions, and Fibonacci sequences}, The Fibonacci Quarterly, Vol.~37, No.~2 (1999), pp.~135--140, \url{https://www.fq.math.ca/Scanned/37-2/gould.pdf}.
\bibitem[GR88]{GR88}
I.~S.~Gradshteyn, I.~M.~Ryzhik, and Robert H. Romer, 
\textit{Tables of integrals, series, and products}, 
American Association of Physics Teachers, 1988.
\bibitem[HR1918]{HR1918}
G.~H. Hardy and S.~Ramanujan,
``Asymptotic formulae in combinatory analysis,''
Proc.\ Lond.\ Math.\ Soc.\ (2) \textbf{17} (1918), 75--115.
\bibitem[Horn2012]{Horn2012}
Roger A. Horn and Charles R. Johnson, \textit{Matrix Analysis}, Cambridge University Press, 2012.
\bibitem[Le47]{Le47} D.~H. Lehmer, \textit{The vanishing of Ramanujan's function $\tau(n)$}, Duke Math.\ J.\ \textbf{14} (1947), 429--433, \url{https://github.com/barry314159a/tau_properties}.
\bibitem[LMFDB]{LMFDB} 
The LMFDB Collaboration,
\textit{The {L}-functions and Modular Forms Database}
  howpublished = {\url{https://www.lmfdb.org}},
  year = {2025}
\bibitem[M]{M} I.~G. Macdonald, \textit{Symmetric Functions and Hall Polynomials}, 2nd ed., Clarendon Press, Oxford, 1995.
\bibitem[OEIS]{OEISA000166}
OEIS Foundation Inc.
\newblock Sequence A000166, Number of permutations of $\{1,2,\ldots,n\}$ with no fixed points (derangements).
\newblock {\em The On-Line Encyclopedia of Integer Sequences}, 2025.
\newblock \texttt{https://oeis.org/A000166}.
\bibitem[Pe]{Pe} A. Petojevi{\'c}, 
\textit{Personal communication},
\url{https://github.com/barry314159a/public-tau-properties}
\bibitem[R2015]{R2015}
S.~Ramanujan, \emph{Collected Papers of Srinivasa Ramanujan}, Cambridge University Press, 2015.
\bibitem[R]{R} S. Ramanujan, \textit{On certain arithmetical functions}, Collected Papers of Srinivasa Ramanujan, AMS Chelsea, 2000.
\bibitem[sagemath]{sagemath}
The Sage Developers,
\newblock \emph{SageMath, the Sage Mathematics Software System (Version 10.8)},
\newblock 2024, \url{https://www.sagemath.org}.
\bibitem[Si]{Si}
J.~H. Silverman,
\textit{The Arithmetic of Elliptic Curves},
Graduate Texts in Mathematics, vol.~106,
Springer-Verlag, New York, 1986 (2nd ed.\ 2009).
\bibitem[SloaneA006922]{SloaneA006922}
N.~J.~A.~Sloane, \emph{The On-Line Encyclopedia of Integer Sequences}, published electronically at \url{https://oeis.org/A006922}, accessed on August 11, 2025.
\bibitem[T]{T} H. W. Turnbull, \textit{Theory of Equations}, Oliver and Boyd, Edinburgh and London, 1947, pp.~66--80.
\bibitem[Vein1999]{Vein1999}
Robert Vein and Paul Dale,
\textit{Determinants and their Applications in Mathematical Physics},
Springer, 1999,
\url{https://www.nzdr.ru/data/media/biblio/kolxoz/M/MA/MAl/Vein%20R.,%20Dale%20P.%20Determinants%20and%20their%20applications%20in%20mathematical%20physics%20(p219-220%20scan)(Springer,%201998)(ISBN%200387985581)(O)(394s)_MAl_.pdf}.
\bibitem[Wilf2006]{Wilf2006}
H.~S.~Wilf.
\newblock {\em generatingfunctionology}, third edition.
\newblock A K Peters, Wellesley, MA, 2006.
\newblock Available at \texttt{https://www2.math.upenn.edu/\~{}wilf/DownldGF.html}.
\bibitem{BCDT}
C.~Breuil, B.~Conrad, F.~Diamond, and R.~Taylor.
\newblock On the modularity of elliptic curves over $\mathbb{Q}$: wild 3-adic exercises.
\newblock \emph{Journal of the American Mathematical Society}, 14(4):843--939, 2001.
\newblock \url{https://doi.org/10.1090/S0894-0347-01-00370-8}

\bibitem{BSD}
B.~J. Birch and H.~P.~F. Swinnerton-Dyer.
\newblock Notes on elliptic curves. II.
\newblock \emph{Journal f\"{u}r die reine und angewandte Mathematik}, 218:79--108, 1965.
\newblock \url{https://doi.org/10.1515/crll.1965.218.79}

\bibitem{cohen:course}
H.~Cohen.
\newblock \emph{A Course in Computational Algebraic Number Theory}.
\newblock Graduate Texts in Mathematics, Vol.~138. Springer, 1993.
\newblock \url{https://doi.org/10.1007/978-3-662-02945-9}

\bibitem{cox:primes}
D.~A. Cox.
\newblock \emph{Primes of the Form $x^2 + ny^2$: Fermat, Class Field Theory, and Complex Multiplication}.
\newblock Wiley, 2nd edition, 2013.

\bibitem{cremona:algorithms}
J.~E. Cremona.
\newblock \emph{Algorithms for Modular Elliptic Curves}.
\newblock Cambridge University Press, 2nd edition, 1997.
\newblock Available at \url{https://homepages.warwick.ac.uk/~masgaj/book/fulltext/}

\bibitem{cremona:database}
J.~E. Cremona.
\newblock \emph{The Elliptic Curve Database for Conductors to 500000}.
\newblock \url{https://johncremona.github.io/ecdata/}, 2024.

\bibitem{dummit-foote}
D.~S. Dummit and R.~M. Foote.
\newblock \emph{Abstract Algebra}.
\newblock Wiley, 3rd edition, 2004.

\bibitem{gross-zagier}
B.~H. Gross and D.~B. Zagier.
\newblock Heegner points and derivatives of L-series.
\newblock \emph{Inventiones mathematicae}, 84(2):225--320, 1986.
\newblock \url{https://doi.org/10.1007/BF01388809}

\bibitem{mazur:torsion}
B.~Mazur.
\newblock Modular curves and the Eisenstein ideal.
\newblock \emph{Publications Math\'{e}matiques de l'IH\'{E}S}, 47:33--186, 1977.
\newblock \url{http://www.numdam.org/item/PMIHES_1977__47__33_0/}

\bibitem{mordell}
L.~J. Mordell.
\newblock On the rational solutions of the indeterminate equations of the third and fourth degrees.
\newblock \emph{Proceedings of the Cambridge Philosophical Society}, 21:179--192, 1922.

\bibitem{serre:galois}
J-P.~Serre.
\newblock Propri\'{e}t\'{e}s galoisiennes des points d'ordre fini des courbes elliptiques.
\newblock \emph{Inventiones mathematicae}, 15(4):259--331, 1972.
\newblock \url{https://doi.org/10.1007/BF01405086}

\bibitem{serre:abelian}
J-P.~Serre.
\newblock \emph{Abelian $\ell$-Adic Representations and Elliptic Curves}.
\newblock A K Peters, 1998. (Originally published 1968)

\bibitem{silverman:aec}
J.~H. Silverman.
\newblock \emph{The Arithmetic of Elliptic Curves}.
\newblock Graduate Texts in Mathematics, Vol.~106. Springer, 2nd edition, 2009.
\newblock \url{https://doi.org/10.1007/978-0-387-09494-6}

\bibitem{silverman:advanced}
J.~H. Silverman.
\newblock \emph{Advanced Topics in the Arithmetic of Elliptic Curves}.
\newblock Graduate Texts in Mathematics, Vol.~151. Springer, 1994.
\newblock \url{https://doi.org/10.1007/978-1-4612-0851-8}

\bibitem{silverman:computing}
J.~H. Silverman.
\newblock Computing heights on elliptic curves.
\newblock \emph{Mathematics of Computation}, 51(183):339--358, 1988.
\newblock \url{https://doi.org/10.1090/S0025-5718-1988-0942161-4}

\bibitem{sutherland:galois}
A.~V. Sutherland.
\newblock Computing images of Galois representations attached to elliptic curves.
\newblock \emph{Forum of Mathematics, Sigma}, 4:e4, 2016.
\newblock \url{https://doi.org/10.1017/fms.2015.33}

\bibitem{tay-wil}
R.~Taylor and A.~Wiles.
\newblock Ring-theoretic properties of certain Hecke algebras.
\newblock \emph{Annals of Mathematics}, 141(3):553--572, 1995.
\newblock \url{https://doi.org/10.2307/2118560}

\bibitem{washington:elliptic}
L.~C. Washington.
\newblock \emph{Elliptic Curves: Number Theory and Cryptography}.
\newblock Chapman \& Hall/CRC, 2nd edition, 2008.

\bibitem{weil}
A.~Weil.
\newblock L'arithm\'{e}tique sur les courbes alg\'{e}briques.
\newblock \emph{Acta Mathematica}, 52:281--315, 1929.
\newblock \url{https://doi.org/10.1007/BF02592188}

\bibitem{wiles:fermat}
A.~Wiles.
\newblock Modular elliptic curves and Fermat's Last Theorem.
\newblock \emph{Annals of Mathematics}, 141(3):443--551, 1995.
\newblock \url{https://doi.org/10.2307/2118559}

\end{thebibliography}
\end{document}